\documentclass{amsart}
\theoremstyle{plain}

\usepackage{bbm}

\usepackage[multiple]{footmisc}
\usepackage{amsthm}
\usepackage{amssymb}
\usepackage{amsmath}
\usepackage{xfrac}
\usepackage{xparse}
\usepackage{faktor}
\usepackage{array}
\usepackage{mathtools}
\usepackage{stmaryrd}
\usepackage{inputenc}
\usepackage{csquotes}

\newtheorem{prop}{Proposition}	
	
\newtheorem{definition}[prop]{Definition}
\newtheorem{theorem}[prop]{Theorem}

\newtheorem{cor}[prop]{Corollary}
\newtheorem{lemma}[prop]{Lemma}
\newtheorem{remark}[prop]{Remark}

\numberwithin{prop}{section}
\numberwithin{equation}{section}

\newtheorem*{definition*}{Definition}
\newtheorem*{theorem*}{Theorem}
\newtheorem*{remark*}{Remark}

\newcommand{\NN}{\ensuremath{\mathbb{N}}}
\newcommand{\ZZ}{\ensuremath{\mathbb{Z}}}

\newcommand{\IRS}[1]{\ensuremath{\mathrm{IRS}\left( #1 \right) }}

\newcommand{\IRSfi}[1]{\ensuremath{\mathrm{IRS}_{\mathrm{fi}}\left( #1 \right) }}
\newcommand{\IRSerg}[1]{\ensuremath{\mathrm{IRS}_{\mathrm{ergodic}}\left( #1 \right) }}
\newcommand{\IRSfe}[1]{\ensuremath{\mathrm{IRS}_{\mathrm{fe}}\left( #1 \right) }}

\DeclarePairedDelimiter{\ceil}{\lceil}{\rceil}
\DeclarePairedDelimiter{\floor}{\lfloor}{\rfloor}

\newcommand{\nrm}{\ensuremath{\vartriangleleft }}

\newcommand{\Sub}[1]{\ensuremath{\mathrm{Sub}\left({#1}\right)}}
\newcommand{\Subd}[2]{\ensuremath{\mathrm{Sub}_{\le #2}\left({#1}\right)}}

\newcommand{\Subfi}[1]{\ensuremath{\mathrm{Sub}_{\mathrm{fi}}\left({#1}\right)}}
\newcommand{\Pow}[1]{\ensuremath{\mathrm{Pow}\left({#1}\right)}}
\newcommand{\Probs}[1]{\ensuremath{\mathcal{P}\left({#1}\right)}}

\title[Infinitely presented permutation stable groups]{Infinitely presented permutation stable groups and invariant random subgroups of metabelian groups}
\author{Arie Levit \and Alexander Lubotzky}

\begin{document}

\maketitle

\begin{abstract} 
We prove that  all invariant random subgroups of the  lamplighter group $L$ are co-sofic.  It follows that $L$ is permutation stable,   providing an  example of an infinitely presented such a group. Our  proof applies  more generally to  all permutational wreath products of finitely generated abelian groups.  We rely on the pointwise ergodic theorem for amenable groups. 
\end{abstract}

\section{Introduction}

 Let $\mathrm{S}(n)$ denote the symmetric group of degree $n \in \NN$ with   the bi-invariant Hamming metric $d_n$  given by
$$ d_n(\sigma, \tau) = 1 - \frac{1}{n} \left|\mathrm{Fix}(\sigma^{-1} \tau)\right|.$$

Let $G$ be a finitely generated group. An \emph{almost-homomorphism} of $G$ is a sequence of set theoretic maps $f_n : G \to \mathrm{S}(n )$   satisfying
$$ d_{n } \left(f_n(g) f_n(h), \, f_n(gh)\right) \xrightarrow{n\to\infty} 0, \quad \forall g,h \in G.$$
The   almost-homomorphism $f_n$ is \emph{close to a homomorphism} if there is a sequence of group homomorphisms $\rho_n : G \to \mathrm{S}(n )$ satisfying
$$ d_{n } \left(\rho_n(g), \, f_n(g)\right) \xrightarrow{n\to\infty} 0 \quad \forall g\in G.$$

\begin{definition}
\label{def:stable in permutations}
The group $G$ is \emph{permutation stable} (or $P$-stable for short) if every almost-homomorphism of $G$ is close to a homomorphism.
\end{definition}

Various   notions of stability have been considered in the literature, cf. \cite{arzhantseva2015almost, de2017stability, ThomICM}. In this paper we consider permutation stability.  In recent years there has been a growing interest in this kind of  stability, originating  in the study of \enquote{almost solutions} to group theoretic equations \cite{glebsky2009almost}. Interestingly, permutation stability provides   a possible approach to tackle the seminal problem \enquote{are there non-sofic groups}? See \cite{arzhantseva2015almost, becker2019stability, bowen2019flexible, glebsky2009almost, ThomICM} and the references therein.

Until quite recently only very few groups were known to be permutation stable: free groups (trivially), finite groups \cite{glebsky2009almost} and abelian groups \cite{arzhantseva2015almost}.  The situation was dramatically changed in \cite{becker2019stability}. That work established a connection between permutation stability and invariant random subgroups
 of a given amenable group $G$. 
 
 An invariant random subgroup is  a conjugation invariant   probability measure on the   space of  all   subgroups of $G$. An invariant random subgroup $\mu$ is  called   \emph{co-sofic} if $\mu$  is a limit   of invariant random subgroups supported on finite index subgroups\footnote{See  \S\ref{sec:random subsets} below for a more detailed discussion      of invariant random subgroups.}.

\begin{theorem}[\cite{becker2019stability}]
\label{thm:BLT}
A finitely generated  amenable group  $G$ is permutation stable  if and only if every invariant random subgroup of $G$ is co-sofic.
\end{theorem}

Proving that   every invariant random subgroup   is co-sofic is also of interest for general groups. For example,  the Aldous--Lyons conjecture \cite{aldous2007processes}   asserts that this is the case for free groups. Another example is the Stuck--Zimmer theorem \cite{stuck1994stabilizers} proving this  for all high rank lattices with property (T). See   \cite{gelander2015lecture} for details.


The significance of Theorem \ref{thm:BLT} is in transforming the question of permutation stability for amenable groups into the realm of invariant random subgroups. This new view point enabled the authors of \cite{becker2019stability} to give many examples of permutation stable groups, e.g. the polycyclic-by-finite  ones as well as the  Baumslag--Solitar groups $B(1,n)$ for all $ n \in \ZZ$. In these examples it is relatively straightforward to classify all invariant random subgroups and in particular show all are  co-sofic. 

If one considers slightly more complicated solvable groups, e.g. the lamplighter group $\ZZ_2 \wr \ZZ$, the answer is no longer immediately clear. A detailed study of the invariant random subgroups of $\ZZ_2 \wr \ZZ$ was performed in \cite{bowen2015invariant, grigorchuk2014lattice, hartman2015furstenberg} but towards different sets of goals. The main goal of this paper is to prove

\begin{theorem}
\label{thm:main theorem for lamplighters}
Every invariant random subgroup of the lamplighter group $\ZZ_2 \wr \ZZ$ is co-sofic and so the group  $\ZZ_2 \wr \ZZ$ is permutation stable.
\end{theorem}



\subsection*{Permutational wreath products}
Our results are in fact more general.
Namely, let $Q$ and $B$ be   finitely generated abelian groups. Let $X$ be a set admitting an action of the group $Q$ with finitely many orbits. The \emph{permutational wreath product} $B \wr_X Q$ corresponding to this $Q$-action and to the base group $B$ is the semidirect product $Q \ltimes \bigoplus_{x \in X} B$. The group $Q$ acts on the normal subgroup $\bigoplus_{x \in X} B$ by group automorphisms permuting coordinates.


\begin{theorem}
	\label{thm:main theorem}
Let $G$ be a  permutational wreath product of two finitely generated abelian groups. Every invariant random subgroup of $G$ is co-sofic, and hence     $G$ is permutation stable.
\end{theorem}

%

The standard wreath product  $G = B \wr Q$ of the two finitely generated abelian groups $B$ and $Q$ is finitely presented if and only if the group $Q$ is finite \cite{de2006finitely}. 
Therefore Theorems \ref{thm:main theorem for lamplighters} and
  \ref{thm:main theorem} provide  the first known examples of infinitely presented  groups which are permutation stable\footnote{After this paper was written and circulated, an updated version of \cite{zheng2019rigid} came out, showing that the Grigorchuk group (which is amenable and infinitely presented) is also permutation stable.}.  In a subsequent paper \cite{levit2019uncountably} we   show that there exist uncountably many non-solvable such groups.

\begin{remark}
The reader may find it easier to assume that $G$ is the lamplighter group $\ZZ_2 \wr \ZZ$ upon the first reading of this paper, and especially so in  \S\ref{sec:approximation in permutational}.
\end{remark}
%

\subsection*{An explicit example}
For  the sake of illustrating Theorem  	\ref{thm:main theorem} we provide an explicit application. 
Consider the wreath product
$$ \ZZ\wr \ZZ \cong \left<b,t \; | \; [b,b^{t^i}] = \mathrm{id} \; \; \forall i \in \NN \right>.$$
The group $\ZZ \wr \ZZ$ is   infinitely presented \cite[p. 241]{lennox2004theory}. It is permutation stable  by our Theorem   \ref{thm:main theorem}.   A hands-on interpretation of this can be formulated as follows.

\begin{cor} 
\label{cor:application to wreath product}
For every $\varepsilon > 0$ there are some $\delta = \delta(\varepsilon) > 0$ and $k = k (\varepsilon) \in \NN$ with the following property:  for every $n \in \NN$ and every pair of permutations $\beta, \tau \in \mathrm{S}(n)$ satisfying
$$ d_n( [\beta, \beta^{\tau^i}  ], \, \mathrm{id} ) < \delta \quad \forall i \in \{1,\ldots, k\}$$
there exist permutations $b, t \in  \mathrm{S}(n)$ satisfying $[b,  b^{t^i} ] = \mathrm{id}$ \emph{for all} $i \in \NN$   with
$$ d_n(b,\beta) < \varepsilon \quad \text{and} \quad d_n(t,\tau) < \varepsilon.$$
\end{cor}

This example illustrates that our results have a very concrete combinatorial meaning. It is perhaps of interest to point out that the proof, and in fact the entire rest of this paper, deals with invariant random subgroups studied  via the methods of ergodic theory. 


\subsection*{Weiss approximated subgroups } 

 Let $G$ be a discrete group and $d$ be any compatible metric on the space of all probability measures on the space of all subgroups of $G$. A  key idea of this paper is    Weiss approximations.

\begin{definition}
A \emph{Weiss approximation} for the subgroup $H$ of $G$ is	a sequence  $(K_i,F_i)$ where for every $i\in\NN$, $K_i  \le G $ is a   finite index subgroup  and $F_i$ is a finite union of disjoint transversals of $\mathrm{N}_G(K_i)$ in $G$   such that
	\begin{equation}
	 \lim_{i\to\infty} d\left(\frac{1}{|F_i|}\sum_{f\in F_i} \delta_{fK_if^{-1}}, \frac{1}{|F_i|}\sum_{f\in F_i} \delta_{fHf^{-1}}\right) = 0.
	\end{equation}
\end{definition}

This notion is valid for any group $G$, but turns out to be especially useful in the amenable case. By using Weiss' work on monotilable   groups \cite{weiss2001monotilable} combined with the Lindenstrauss      point-wise ergodic theorem   for amenable group \cite{lindenstrauss2001pointwise} we show: 

\begin{theorem}
\label{thm:cosofic iff weiss}
Let $G$ be a finitely generated residually finite amenable group. Then $G$ admits a  sequence of finite subsets $F_i$  with the following property: an ergodic invariant random subgroup $\mu$ of  $G$ is co-sofic if and only if $\mu$-almost every subgroup $H$ is Weiss approximated by $(K_i,F_i)$ for some sequence $K_i$ of finite index subgroups.
\end{theorem}

The above result follows as a  combination of our  two  Theorems \ref{thm:cosofic subgroup implies cosofic IRS} and \ref{thm:converse on cosofic for residually finite} below.  We are then able to use Theorem \ref{thm:cosofic iff weiss} towards our main result by constructing suitable  Weiss approximations.

\subsection*{Outline of the paper}
Let us briefly outline the  structure of the paper.
 In \S\ref{sec:random subsets} we recall the notion of an invariant random subgroup of   $G$ and describe some of its properties. This is followed by \S\ref{sec:weiss approximable subgroups} where we introduce  and discuss Weiss approximations. In \S\ref{sec:Goursat} we study Folner sequences and transversals in group extensions.

The rest of the paper  is dedicated to a detailed  analysis of the subgroups and invariant random subgroups of metabelian groups. The results of   \S\ref{sec:subs and irs of metabelian} and \S\ref{sec:controlled approximations}
  hold true for all metabelian groups, while starting from \S\ref{sec:permutational metabelian groups} we restrict our attention to what we call   \enquote{permutational metabelian groups}. 
  
The   construction of   Weiss approximations is performed   in \S\ref{sec:approximation in permutational} only for the split case. It amounts to the careful construction of    Folner sequences of transversals that satisfify a certain statistical property of being \enquote{adapted}, see Definition \ref{def:adapted sequence}.

\subsection*{Open problems and questions}

We have enough evidence to believe that Theorem \ref{thm:main theorem} could be true for all permutational metabelian groups. This would imply, for instance, that finitely generated free metabelian groups are permutation stable. 

The classical result of Phillip Hall \cite{hall1959finiteness} asserts that every metabelian group is residually finite. Wishful thinking seems to  suggest  that all metabelian groups are permutation stable. If true, this   would be a far reaching strengthening of Hall's theorem. Note that the smallest residually finite solvable group known  to be non permutation stable   is the $3$-step solvable Abels' group \cite{becker2019stability}.   

More generally, a permutational wreath product $B \wr_X Q$ of two arbitrary groups is residually finite if and only if both $Q$ and $B$ are residually finite and either $B$ is abelian or $Q$ is finite \cite{grunberg1957residual}.
A wreath product  of two amenable groups is again amenable and hence sofic   \cite{weiss2000sofic}. A sofic group which is    permutation stable   must be residually finite   \cite{glebsky2009almost}.
Therefore if $B \wr_X Q$ is sofic and permutation stable then, unless $Q$ is finite, the base group $B$ must be abelian.  
 The moral  is that   while our Theorem \ref{thm:main theorem} might not be as general as possible in the sense that $Q$ could potentially be replaced by other amenable groups,  at the very least    the base group $B$  has to be abelian.
 
The  wreath product of two finitely generated abelian groups is known to be locally extended residually finite (LERF) by \cite{alperin2006metabelian}. Therefore our Theorem \ref{thm:main theorem}   gives a partial answer in the direction suggested by \cite[Question 8.6]{becker2019stability}, namely is every LERF amenable group permutation stable?

\vspace{10pt}
\emph{We follow the convention $x^y = y^{-1} x y$ for conjugation and $[x,y]  = x^y x^{-1}$ for the commutator. The following  two identities
 \begin{equation}
 \label{eq:standard commutator identities}
  \left[xy,z\right] = \left[x,z\right] ^y \left[y,z\right] \quad \text{and} \quad  \left[x,yz\right] = \left[x,z\right]\left[x,y\right]^z\end{equation}
will be frequently   used throughout this text.}

\subsection*{Acknowledgements}
The authors would like to thank Oren Becker for stimulating discussions and in particular for bringing to our attention Corollary \ref{cor:residually finite amenable has Foler sequence of finite to one transverals}. 

The first named author is indebted for support from the NSF. The second named author   is indebted for support from the NSF, the ISF and the European
Research Council (ERC) under the European Union's Horizon 2020
research and innovation program (grant agreement No. 692854).
%
%
%


\section{Random subsets and subgroups}
\label{sec:random subsets}

Let $G$ be a countable discrete group. We recall  the Chabauty space   of all subgroups of $G$ and   define invariant random subgroups of $G$. We then investigate the notion of co-sofic invariant random subgroups in some detail.

\subsection*{Space of subsets}

Consider the \emph{power set} of $G$
$$ \Pow{G} = \{0,1\}^\Gamma$$ 
equipped with the Tychonoff product topology. The space $\Pow{G}$ is compact and metrizable. Indeed, given an arbitrary enumeration $G = \{g_1, g_2, \ldots\}$, the  metric $d_{\Pow{G}}$ given by
\begin{equation}
 d_{\Pow{G}}(A,B) = \sum_{n=1}^\infty \frac{\mathrm{1}_{A \triangle B}(g_n)}{2^n} 
 \end{equation}
for every two subsets $A,B \in \Pow{G}$ induces  the topology. 

%

The group $G$ acts on its power set  $\Pow{G}$ by homeomorphisms via conjugation. We denote this action by $c_g$. Given an element $g \in G$ and a subset $A \subset G$ denote
$$c_g A = A^{g^{-1}} =   gAg^{-1}. $$

\begin{definition}
	\label{def:exhausting sequence}
A sequence $F_i \in \Pow{G}$  is \emph{exhausting} if     $F_i$ Chabauty converges to the point $G \in \Pow{G}$.
\end{definition}

\subsection*{Chabauty space of subgroups}

Consider the following subset of $\Pow{G}$
$$\Sub{G} = \{H \le G \: : \: \text{$H$ is a subgroup of $G$} \}.$$ 
The space $\Sub{G}$ is called the \emph{Chabauty space} of the group $G$. The space $\Sub{G}$ is compact since it is a  closed subset of $\Pow{G}$. Moreover $\Sub{G}$ is preserved by the conjugation action $c_g$ of $G$.

Let $\Subd{G}{d}$ denote the subset of $\Sub{G}$ consisting of all subgroups of index at most $d$ in $G$. Denote 
$$\Subfi{G} = \bigcup_{d \in \NN}\Subd{G}{d}$$ so that $\Subfi{G}$ consists of all the finite index subgroups of $G$. Of course, the subset $\Subfi{G}$ need not  be closed in general.


\subsection*{Spaces of probability measures} 

Let $\Probs{G}$ be the space of all Borel probability measures on the   set $\Sub{G}$. This is a compact   space with the weak-$*$ topology according to the Banach--Alaoglu theorem. The  conjugation action  of $G$ on its Chabauty space  $\Sub{G}$ gives rise to a corresponding push-forward action  of $G$ on the space $\Probs{G}$. We continue using the notation  $c_g$ for this push-forward action. 
\begin{lemma}
	\label{lem:lemma on IRS convergence}
	A sequence $\mu_n \in \Probs{G}$   weak-$*$ converges to  $ \mu \in \Probs{G}$ if and only if for every pair of  finite subsets $A,B \subset G$     we have    
	\begin{equation}
	 \mu_n(E_{A,B}) \xrightarrow{n\to\infty} \mu(E_{A,B}) 
	 \end{equation}
	where
	\begin{equation}
	E_{A,B} = \{ H \le G \: : \: \text{$A \subset H$ and $B \subset G\setminus H$}  \}. 
	\end{equation}
\end{lemma}
\begin{proof}
Given a pair of finite subsets $A,B \subset G$, the subset $E_{A,B} $  is the intersection of a \enquote{cylinder set} in the Tychonoff product space $\Pow{G}$ with the Chabauty subspace  $ \mathrm{Sub}(G)$. Therefore $E_{A,B}$ is clopen.  According to the Portmanteau theorem \cite[Theorem 2.1]{billingsley2013convergence}, the weak-$*$ convergence $\mu_n \to \mu$ implies   $\mu_n(E_{A,B}) \to \mu(E_{A,B})$. 

The family of all subsets $E_{A,B}$ with $A$ and $B$ ranging over finite subsets of the group $G$ is a countable basis for the topology of $\mathrm{Sub}(G)$. In addition, the family of subsets $E_{A,B}$ is closed under finite intersections. These are precisely the two conditions needed for   \cite[Theorem 2.2]{billingsley2013convergence} which amounts to the converse direction of our lemma. The proof of that theorem is essentially an inclusion-exclusion formula.
\end{proof}

The space $\Probs{G}$ is metrizable. In fact, it follows from Lemma \ref{lem:lemma on IRS convergence} that if $A_i$ and $ B_i$ is an arbitrary  enumeration of all pairs of   finite subsets of $G$ then 
\begin{equation}
\label{eq:metric on P}
 d_{\Probs{G}}(\mu,\nu) = \sum_{i \in \NN}\frac{|\mu(E_{A_i,B_i}) - \nu(E_{A_i,B_i})|}{2^i} \quad  \mu,\nu \in \Probs{G}
 \end{equation}
is a compatible metric for the topology of $\Probs{G}$.

\subsection*{Invariant random subgroups}
Denote
$$ \IRS{G} = \{ \mu \in \Probs{G} \: : \: \text{ $c_g \mu = \mu$ for all $g \in G$} \}.$$
Note that $\IRS{G}$ is a weak-$*$ closed and hence a compact subset of $\Probs{G}$. An element $\mu \in \IRS{G}$ is called an \emph{invariant random subgroup}  \cite{abert2014kesten}.

An invariant random subgroup $\mu \in \IRS{G}$ is \emph{ergodic} if every convex combination $\mu = t\mu_1 + (1-t)\mu_2$ with $\mu_1,\mu_2 \in \IRS{G}$ and $t \in \left(0,1\right)$ is trivial in the sense that $\mu_1 = \mu_2$. In other words, the ergodic invariant random subgroups $\IRSerg{G}$ are the extreme points of the compact convex set $\IRS{G}$. 

The compact convex space $\IRS{G}$ is a \emph{Choquet simplex}. This means that every $\mu \in \IRS{G}$ admits an \emph{ergodic decomposition}, namely   there is a unique probability measure $\nu_\mu$ on the Borel set $\IRSerg{G}$ satisfying
\begin{equation}
\label{eq:ergodic dec}
 \mu(f)  = \int_{\IRSerg{G}} f \; \mathrm{d} \nu_\mu
 \end{equation}
for every continuous function $ f$ on the compact   space $\Sub{G}$. To interpret this formula 
regard $f$ as a weak-$*$ continuous function on $\IRS{G}$. The left-hand side of   Equation (\ref{eq:ergodic dec}) now reads $f(\mu)$ while the right-hand side reads $\nu_\mu(f)$. 
 The reader is referred to  \cite[\S12]{phelps2001lectures} for details. 

\subsection*{Co-sofic invariant random subgroups} Let $\IRSfi{G}$ denote the subspace consisting of all $\mu \in \IRS{G}$ satisfying $\mu(\Subfi{G}) = 1$. Clearly $\IRSfi{G}$ is convex. Note that $\mu \in \IRSfi{G}$ is ergodic if and only if $\mu$ is a uniform atomic probability measure supported on the conjugacy class of some finite index subgroup $H \le G$.

\begin{definition}
\label{def:cosofic IRS}
The invariant random subgroup $\mu \in \IRS{G}$ is \emph{co-sofic} if 
$$\mu \in \overline{\IRSfi{G}}^{\text{weak-$*$}}.$$
\end{definition}

The collection   of all co-sofic invariant random subgroups is convex and weak-$*$ closed, being the closure of a convex set. 

The  important notion of \emph{sofic groups} was  introduced by Gromov in  \cite{gromov1999endomorphisms}. The name \emph{sofic} was coined by Weiss in \cite{weiss2000sofic}. For our purposes, it will suffice to state an equivalent definition, given in terms of invariant random subgroups.     If $G$ is a free group and $N \nrm G$ is a normal subgroup then the   quotient group $G/N$ is \emph{sofic} if and only if the invariant random subgroup $\delta_N$ is co-sofic.  See   \cite[\S2.5]{gelander2015lecture} for details about this fact. Every amenable   as well as every residually finite group is sofic.
 
 \begin{remark}
 \label{remark:cosofic}
Let $G$ be any group with normal subgroup $N$. 
One direction of the above statement holds true in general. Namely, if   $\delta_N$ is co-sofic then $G/N$ is sofic. This can be seen by lifting  to the free group.
 The converse implication   is false in general. For example, the invariant random subgroup $\delta_{\{e\}}$ is   co-sofic if and only if $G$ is residually finite,  and some sofic groups are not residually finite.  However, in the special case where $G/N$ is residually finite then $\delta_N$ is indeed co-sofic.
 \end{remark}
 
The following result shows that when studying co-sofic invariant random subgroups we may in some sense restrict our attention to ergodic ones.



\begin{prop}
\label{prop:enough to prove co-sofic for ergodic}
 Let $\mu \in \IRS{G}$ be an invariant random subgroup with ergodic decomposition   $\nu_\mu \in \mathcal{P}(\IRSerg{G})$.   If $\nu_\mu$-almost every $\lambda \in \IRSerg{G}$ is co-sofic then     $\mu$ is co-sofic.
\end{prop}

\begin{proof}

Assume that $\nu_\mu$-almost every $\lambda \in \IRSerg{G}$ is co-sofic, so that
 $$ \mathrm{supp}(\nu_\mu) \subset \overline{\IRSfi{G}}^{\text{weak-$*$}}.$$
 Since the collection of all co-sofic invariant random subgroups is closed and convex we deduce that
 $$ \mu \in \overline{\mathrm{conv}(\mathrm{supp}(\nu_\mu))}^{\text{weak-$*$}} \subset \overline{\IRSfi{G}}^{\text{weak-$*$}}$$
 as required.
\end{proof}

The following  observation follows immediately from Proposition \ref{prop:enough to prove co-sofic for ergodic}.

\begin{cor}
\label{cor:enough to prove co-sofic for ergodic}
If    every    $\mu \in \IRSerg{G}$ is co-sofic then every $\mu \in \IRS{G}$ is co-sofic as well.
\end{cor}



Ergodic co-sofic invariant random subgroups admit a very explicit description.

\begin{prop}
\label{prop:if ergodic weak-* limit then is weak* limit of ergodic}
Let  $\mu \in \IRS {G}$ be ergodic and co-sofic. Then $\mu = \lim \nu_i$ for some ergodic $\nu_i \in \IRSfi{G}$.
\end{prop}

\begin{proof}
To ease our notations denote $\IRSfe{G} = \IRSfi{G} \cap \IRSerg{G}$. The invariant random subgroup $\mu$ is co-sofic. Therefore we may write
\begin{equation} 
\label{eq:cosofic and fi}
\mu = \lim_i \mu_i   \quad \text{where} \quad \mu_i \in  \IRSfi{G} = \mathrm{conv} \left( \IRSfe{G} \right)  \quad \forall i\in\NN. 
\end{equation}

Assume towards contradiction that 
$$ \mu \notin \overline{\IRSfe{G}}^{\text{weak-$*$}}.$$
By  the definition of weak-$*$ convergence, this means that there is some real-valued continuous function $f \in \mathrm{C}(\Sub{G})$ and some $\varepsilon > 0$ such that 
$$ |\mu(f) - \lambda(f)| > \varepsilon \quad \forall \lambda \in \IRSfe{G}.$$
Consider the decomposition $ \IRSfe{G} = F_+ \cup F_- $  into two disjoint Borel sets with
$$ \lambda(f) > \mu(f) + \varepsilon \quad \forall \lambda \in F_+ \quad \text{and} \quad \lambda(f) < \mu(f) - \varepsilon \quad \forall \lambda \in F_-.
$$
Note that
\begin{equation}
\label{eq:not in conv of pair}
\mu \notin \overline{\mathrm{conv}(F_+)}^{\text{weak-$*$}} \cup \overline{\mathrm{conv}(F_-)}^{\text{weak-$*$}}.
\end{equation}
The two Equations (\ref{eq:cosofic and fi}) and (\ref{eq:not in conv of pair}) put together  imply that both subsets $F_+$ and $F_-$ are non-empty.

The above decomposition of the set $\IRSfe{G}$ allows us to express each invariant random subgroup $\mu_i$  as a convex combination in the following way
$$ \mu_i \in \mathrm{conv}(F_+ \cup F_-) = \mathrm{conv}(\mathrm{conv}(F_+) \cup \mathrm{conv}(F_-)) \quad \forall i \in \mathbb{N}.$$
Taking the weak-$*$ limit $\mu = \lim \mu_i$ gives
\begin{equation}
\label{eq:another expression for mu}
 \mu \in \mathrm{conv}(\overline{\mathrm{conv}(F_+)}^{\text{weak-$*$}} \cup \overline{\mathrm{conv}(F_-)}^{\text{weak-$*$}}).
 \end{equation}
Putting together the two Equations (\ref{eq:not in conv of pair}) and (\ref{eq:another expression for mu}) implies that $\mu$ is necessarily  a non-trivial convex combination. This is a contradiction to the ergodicity of $\mu$.
%
%
\end{proof}

%

\subsection*{Chabauty spaces of subgroups and quotients}

Let $H$ be any subgroup of $G$. There is a continuous \emph{restriction map}  given by
$$ \cdot_{|H} : \Pow{G} \to \Pow{H}, \quad A \mapsto A_{|H} = A \cap H \quad \forall A \subset G. $$
It is clear that $ \Sub{G}_{|H} = \Sub{H}$. Similarly
$$ \Subd{G}{d}_{|H} \subset \Subd{H}{d} \quad \forall d \in \NN \quad \text{and} \quad \Subfi{G}_{|H} \subset \Subfi{H}.$$
The restriction map is $H$-equivariant for the conjugation action   of the subgroup $H$.
Pushing-forward via the restriction determines a map
$$ \cdot_{|H} : \Probs{G} \to \Probs{H}, \quad \mu \mapsto \mu_{|H} \quad \forall \mu \in \Probs{G}.$$

\begin{prop}
\label{prop:restriction of an IRS is an IRS}
If $H$ is a subgroup of $G$ then $\IRS{G}_{|H} \subset \IRS{H}$.
\end{prop}
\begin{proof}
This follows immediately from the fact that restriction to $H$ is equivariant with respect to the conjugation action of $H$.
\end{proof}


%
Let $Q$ be a quotient   of $G$ admitting   a surjective homomorphism $\pi : G \to Q$. There is a corresponding map $\pi : \Sub{G} \to \Sub{Q}$ of subgroups taking every subgroup $H \le G$ to its image $\pi(H)$ in $Q$.

\begin{prop}
\label{prop:pushing forward invariant random subgroups}
The map $\pi : \Sub{G} \to \Sub{Q}$  is $G$-equivariant and Borel measurable.
\end{prop}
\begin{proof}
The group $G$ is acting on  $\Sub{G}$ and on $\Sub{Q}$ by conjugation. The $G$-equivariance of the map $\pi$ is clear. To see that $\pi$ is Borel measurable observe that the Chabauty topology on $\Sub{Q}$ is generated by the subsets 
$$ S_q = \{ L \le Q \: : \: q \in L \}$$ 
and their complements. It is easy to verify that the preimage $\pi ^{-1}(S_q)$ is Borel in $\Sub{G}$ for every $q \in Q$.
\end{proof}

In particular, if $\mu$ is an   invariant random subgroup of $G$ then $\pi_* \mu$ is an   invariant random subgroup of $Q$, and if $\mu$ is ergodic then so is $\pi_* \mu$.

\section{Weiss approximable subgroups}
\label{sec:weiss approximable subgroups}

We introduce a notion of Weiss approximable subgroups and relate this to co-soficity of invariant random subgroups. Weiss approximation is inspired by Benjamin Weiss' work \cite{weiss2001monotilable}.
 Our arguments   in this section crucially rely on the pointwise ergodic theorem for amenable groups due to Lindenstrauss \cite{lindenstrauss2001pointwise}.

 \subsection*{Transversals and finite index subgroups}

Let $G$ be a   discrete  group and $H$ be a fixed subgroup of $G$. 
\begin{definition}
\label{def:transversal}
A \emph{(left) transversal} for $H$ in $G$ is a subset $T_0$ consisting of one element from each left coset $tH$ of $H$ in $G$. A \emph{finite-to-one (left) transveral} $T$ for $H$ in $G$ is a disjoint union of finitely many left transversals for $H$ in $G$.
\end{definition}

That is to say $T$ is a finite-to-one   transversal for $H$ in $G$ if there is some $ k\in \NN$ so that $T$ consists of   exactly $k$ elements from each   coset of $H$.

\begin{prop}
\label{prop:universal sequence}
If $G$ is finitely generated then there is a countable   family $\{T_i\}_{i\in\NN}$ of finite subsets such that every finite index subgroup $H$ admits   $T_i$ as a finite-to-one transversal  for all indices $i \in \NN$ sufficiently large.
\end{prop}
\begin{proof}
For every $i \in \NN$ let $N_i$ be the intersection of all subgroups of $G$ with index at most $i$. Since $G$ is finitely generated there are only finitely many such subgroups and therefore   $N_i$ is a finite index normal subgroup of $G$. Let $T_i$ be any transversal to the subgroup $N_i$ in the group $G$. It follows that $T_i$ is a finite-to-one transversal to any finite index subgroup $H$ satisfying $\left[G:H\right] \le i$.
\end{proof}

We say that  such $\{T_i\}_{i\in\NN}$ is a \emph{universal sequence of  transversals} in the group $G$.
 
\subsection*{Actions of finite sets} 
Let $F \subset G$ be a finite subset. Given any subgroup $H \le G$ we will    let $F * H$ denote the   probability measure 
\begin{equation}
\label{eq:conv}
 F * H = \frac{1}{|F|}\sum_{f \in F} c_f \delta_{H}
 \end{equation}
regarded as a point in  $\Probs{G}$. Recall that $c_f \delta_H = \delta_{fHf^{-1}} = \delta_{H^{f^{-1}}}$.

\begin{prop}
\label{prop:finite index subgroups and transversals}
If $\left[G: \mathrm{N}_G(H) \right] < \infty$ and $T$ is \emph{any} finite-to-one transversal of the subgroup $\mathrm{N}_G(H)$   then 
$ T * H \in \IRSfi{G}$ and $T * H$ is supported on the finite conjugacy class of $H$ in $G$.
\end{prop}
\begin{proof}
Let $T$ be any finite-to-one left transversal in $G$ of the subgroup $\mathrm{N}_G(H)$. Let   $k  \in \NN$ be such that $T = \coprod_{j=1}^{k} T_{j}$ where each $T_{j}$ is a left one-to-one transversal of $\mathrm{N}_G(H) $.  Let $\mu_H \in \IRS{G}$ be the invariant random subgroup of $G$ supported on the finite family of subgroups conjugate to   $H$. Note that
$$\mu_H = \frac{1}{k}\sum_{j=1}^{k} \mu_H  = \frac{1}{k}\sum_{j=1}^{k} \frac{1}{\left[G:\mathrm{N}_G(H)\right]} \sum_{g \in T_{j}} \delta_{H^{g^{-1}}}  = \frac{1}{|T|}\sum_{g \in T} c_g \delta_{H} = T * H.$$ 
as required.
\end{proof}

 

\subsection*{Weiss approximation}

Let $H$ be a subgroup of the discrete group $G$.

\begin{definition}
\label{def:cosofic subgroup}

The subgroup $H$ is \emph{Weiss approximable} in $G$ if there are finite index subgroups $K_i$ of $G$ with finite-to-one transversals $F_i$ to $\mathrm{N}_G(K_i)$ such that
$$ d_{\Probs{G}}(F_i * K_i, F_i * H) \xrightarrow{i \to\infty} 0 $$
where $d_{\Probs{G}}$ is any compatible metric on the space $\Probs{G}$. We will say that the sequence $(K_i, F_i)$   is a \emph{Weiss approximation} for $H$ in $G$. 
\end{definition}
The two sequences   $F_i * K_i$ and $ F_i * H$ are not required to converge.  However, by  the compactness of the space  $\Probs{G}$ it is always possible to  pass to a subsequence such that the limits  of $F_i * K_i$ and of $F_i * H$ do exist and coincide in $\Probs{G}$.

\begin{prop}
\label{prop:co-sofic subgroup implies co-sofic IRS}
Let   $(K_i, F_i)$ be a Weiss approximation for the subgroup  $H$ of $G$. If $ F_i * H $ weak-$*$ converges to some $\mu \in \Probs{G}$ then $\mu \in \IRS{G}$ and $\mu$ is co-sofic.
\end{prop}
\begin{proof}
Up to passing to a subsequence we may assume that the weak-$*$ limit of the sequence $\mu_i = F_i * K_i$  in $\Probs{G}$ exists and coincides with   $\mu = \lim F_i * H$.  In addition $\mu_i \in \IRSfi{G}$ by Proposition \ref{prop:finite index subgroups and transversals}. As $ \IRS{G}$ is weak-$*$ closed we have that $\mu \in \IRS{G}$.
We conclude that $\mu$ is co-sofic.
\end{proof}

\begin{cor}
\label{cor:normal subgroups Weiss equals co-sofic}
 A normal subgroup $N \nrm G$ is Weiss approximable if and only if the invariant random subgroup $\delta_N\in\IRS{G}$ is co-sofic.
\end{cor}
\begin{proof}
Let $N \nrm G$ be a normal subgroup. It is clear that $F * \delta_N = \delta_N$ for any finite subset $F \subset G$. If $N$ is Weiss approximable then  Proposition \ref{prop:co-sofic subgroup implies co-sofic IRS} shows that   $\delta_N$ is co-sofic.
Conversely, if $\delta_N$ is co-sofic then we obtain a Weiss approximation for $N$ by  combining       Propositions  \ref{prop:if ergodic weak-* limit then is weak* limit of ergodic} and \ref{prop:finite index subgroups and transversals}.
\end{proof}

 The question of co-soficity  for   invariant random subgroups of the form $\delta_N$ was discussed in  Remark  \ref{remark:cosofic}.  In particular, if $G$ is any discrete group and $N$ is a normal subgroup such that $G/N$ is residually finite then $N$ is Weiss approximable in $G$.

 
 
 \vspace{5pt}
 
We now present an   explicit condition for a given sequence $(K_i, F_i)$   to constitute a Weiss approximation for the subgroup $H$.

\begin{prop}
\label{prop:a numberical condition for a group to be co-sofic}
Let $H$ be a subgroup of $G$. If  a given sequence  $K_i$ of finite index subgroups of $G$ with   finite-to-one transversals $F_i$ of $\mathrm{N}_G(K_i)$ satisfies
$$ p_i(g) = \frac{|\{f \in F_i \: : \: g^f \in K_i \triangle H \}| }{|F_i|} \xrightarrow{i \to \infty} 0 $$
for all elements $g \in G$ then $(K_i,F_i)$ is  a Weiss approximation for the subgroup $H$.
\end{prop}
\begin{proof}
Consider a pair    $A$ and $B$ of   finite subsets of $G$. Denote as before
 $$E_{A,B} = \{ L \le G \: : \: \text{$A \subset L$ and $B \subset G\setminus L$}  \} $$
so that $E_{A,B}$ is a clopen subset  of the Chabauty space $\Sub{G}$. By definition
$$    (F_i * K_i - F_i * H)(E_{A,B})  = \frac{| \{f \in F_i \: : \: c_fK_i  \in E_{A,B}\}| - |\{f \in F_i \: : \: c_fH \in E_{A,B}\}| }{|F_i|}. 
$$
Let $\chi_{ A,B}$ denote the characteristic function of the subset $E_{A,B}$.    An element $f \in F_i$ has a non-zero contribution  to the numerator of the above expression if and only if $\chi_{A,B}(c_fH) \neq \chi_{A,B}(c_fK_i)$. This happens precisely whenever $(A \cup B)^f \cap (K_i \triangle H) \neq \emptyset$. 
We obtain the following upper bound
\begin{align*}
 | (F_i * K_i - F_i * H)(E_{A,B}) |  &\le \frac{ | \{  f \in F_i \: : \: (A \cup B)^f \cap (K_i \triangle H) \neq \emptyset \} |   }{|F_i|} \le \\
 &\le   \sum_{g \in A \cup B} \frac{ | \{  f \in F_i \: : \: g^f \in  K_i \triangle H  \} |   }{|F_i|}  =   \sum_{g \in A \cup B} p_i(g).
\end{align*}
The assumption  that  $p_i(g) \to 0$  for all elements $g \in G$ gives
\begin{equation}
\label{eq:convergence to zero of difference}\lim_{i \to \infty} |(F_i * K_i - F_i * H)(E_{A,B})| = 0
\end{equation} 
for every pair of   finite subsets $A$ and $B$ of $G$. 

Recall Equation $(\ref{eq:metric on P})$ introduced in the discussion following Lemma 	\ref{lem:lemma on IRS convergence}. This is an infinite series expression defining an explicit compatible metric $d_{\Probs{G}}$. All terms of that series are uniformly bounded in the range $\left[0,1\right]$. Therefore Equation (\ref{eq:convergence to zero of difference}) implies $d_{\Probs{G}}(F_i * K_i,F_i * H) \to 0$.
\end{proof}

\subsection*{The pointwise ergodic theorem for amenable groups}  

\begin{definition}
\label{def:Folner sequence}
	A \emph{(left) Folner sequence} for $G$ is a sequence of finite subsets $F_i$ of $G$ such that 
	\begin{equation}
	 \frac{|g F_i \triangle F_i |}{|F_i|} \to 0 
	 \end{equation}
for every element $g \in G$.
\end{definition}
Recall that a Folner sequence $F_i$ is \emph{tempered} if there is a constant $c> 0$ so that
$$| \bigcup_{j<i}F_j^{-1}F_i| \le c|F_i|. $$ 
The tempered  condition is needed for the ergodic theorem for amenable groups, at least in the general case. However, every Folner sequence admits a tempered subsequence \cite[  1.4]{lindenstrauss2001pointwise}. For this reason the tempered condition will not be a major  issue   from our   point of view.

\begin{theorem}[Lindenstrauss  \cite{lindenstrauss2001pointwise}]
\label{thm:pointwise ergodic for amenable}
Let $G$ be an amenable group acting by homeomorphisms on a Hausdorff  second countable   compact space  $X$    ergodically with invariant probability measure $\mu$. Let $F_i$ be a tempered Folner sequence for $G$. Then for  $\mu$-almost every point $x \in X$ we have
\begin{equation}
 \frac{1}{|F_{i}|} \sum_{g \in F_{i}} \delta_{gx} \xrightarrow{i \to\infty} \mu 
 \end{equation}
in the weak-$*$ topology.
\end{theorem}
For the reader's convenience let us briefly recall the well-known argument to derive the above statement as an immediate consequence of \cite[Theorem 1.2]{lindenstrauss2001pointwise}.

 \begin{proof}[Proof of Theorem \ref{thm:pointwise ergodic for amenable}]
It follows from      \cite{lindenstrauss2001pointwise}   that  for every function $f \in L^1(X,\mu)$ 
\begin{equation}
\label{eq:pw}
 \frac{1}{|F_i|} \sum_{g \in F_i} f(gx) \xrightarrow{i \to \infty} \int f \; \mathrm{d} \mu
 \end{equation}
holds true for every point $x$ belonging to some $\mu$-conull subset $X_f \subset X$.

 Let $\mathcal{F}$ be a countable $L^\infty$-dense subset of $C(X)$. Such a family $\mathcal{F}$ exists by the  Stone--Weierstrass theorem. It is clear that $\mathcal{F} \subset L^1(X,\mu)$. So   Equation \ref{eq:pw} holds for every function $f \in \mathcal{F}$ and every point in the $\mu$-conull subset $\bigcap_{f \in \mathcal{F}} X_f$. In particular, for every such point the sequence of probability measures $ \frac{1}{|F_{i}|} \sum_{g \in F_{i}} \delta_{gx}$ weak-$*$ converges to the ergodic probability measure $\mu$.
 \end{proof}

\subsection*{The ergodic theorem and co-sofic subgroups}
 
The following result will be our main tool in establishing that a given invariant random subgroup is co-sofic. 
 
\begin{theorem}
\label{thm:cosofic subgroup implies cosofic IRS}
Let $G$ be an amenable group and $\mu \in \IRS{G}$.  Let $F_i$ be a fixed   Folner sequence in the group $G$. If   $\mu$-almost every subgroup $H$ admits a Weiss approximation $(K_i, F_i)$ with some sequence   of finite index subgroups $K_i$ of $G$  then $\mu$ is co-sofic.
\end{theorem}

\begin{proof}
 In light of Proposition \ref{prop:enough to prove co-sofic for ergodic} we may assume without loss of generality that $\mu$ is ergodic. 
Consider the probability measure preserving  action of the amenable group $G$ on the Borel space $(\Sub{G},\mu)$.

We may assume, up to passing to a subsequence,  that the Folner sequence $F_i$ is tempered. Since a subsequence of Weiss approximation is again a Weiss approximation, the assumption of the Theorem continues to hold.

Let $X_1 \subset \Sub{G}$ be the subset consisting of all subgroups $H \in \Sub{G}$ for which  $\frac{1}{|F_{i}|} \sum_{f \in F_{i}} \delta_{c_f H}$ converges to the invariant random subgroup $\mu$ as $i \to \infty$.    By Theorem \ref{thm:pointwise ergodic for amenable} the subset $X_1$ is $\mu$-measurable and satisfies  $\mu(X_1) = 1$.  

The subset $X_2 \subset \Sub{G}$ consisting of all subgroups $H \in \Sub{G}$ which are Weiss approximable  with respect to some sequence $(K_i, F_i)$  is $\mu$-measurable and satisfies $\mu(X_2)= 1$ by our assumption.
Therefore $X_1 \cap X_2 \neq \emptyset$. Let $H \in \Sub{G}$ be any subgroup belonging to $X_1 \cap X_2$.   The result   follows from Proposition \ref{prop:co-sofic subgroup implies co-sofic IRS}.
\end{proof}

\begin{remark}
There is nothing special about amenable groups in Theorem \ref{thm:cosofic subgroup implies cosofic IRS}. In principle  the same argument  applies to any group $G$ and sequence of subsets $F_i$ satisfying a pointwise ergodic theorem.  
Examples of such situations include the case where $G$ is hyperbolic and $F_i$ is a sequence of balls \cite{bowen2013neumann} and the case of $S$-arithmetic lattices \cite{gorodnik2009ergodic}.
\end{remark}

For residually finite amenable groups Theorem \ref{thm:cosofic subgroup implies cosofic IRS} admits a converse. We include it below for the sake of completeness, even though it will not be used towards our main results.
This   will rely on  the following interesting result due to Weiss \cite{weiss2001monotilable}, see also \cite[Theorem 6]{abert2007rank}.
\begin{theorem}[Weiss]
\label{thm:Weiss}
Let $G$ be a residually finite amenable group. Let $N_i$ be a descending sequence   of finite index normal subgroups of $G$ satisfying $\bigcap_i N_i = \{e\}$. Then there exists a Folner sequence $F_i$ of transversals of the subgroups $N_i$.
\end{theorem}

The notion of a universal sequence of transversals was introduced following Proposition \ref{prop:universal sequence}.

\begin{cor}
\label{cor:residually finite amenable has Foler sequence of finite to one transverals}
If $G$ is a finitely generated  residually finite amenable group then $G$ admits a  Folner sequence of universal transversals.
\end{cor}
\begin{proof}
The Corollary follows exactly as in the proof of Proposition \ref{prop:universal sequence}  while relying   on  Weiss Theorem \ref{thm:Weiss} in the choice of  the transversals.
\end{proof}

We are ready to present the promised converse to Theorem \ref{thm:cosofic subgroup implies cosofic IRS}.

\begin{theorem}
\label{thm:converse on cosofic for residually finite}
Let $G$ be a finitely generated residually finite amenable group. There exists  a   Folner sequence in the group $G$ such that if  $\mu \in \IRS{G}$ is ergodic and co-sofic then $\mu$-almost every subgroup $H$ of $G$ admits a Weiss approximation $(K_i, F_i)$ for some sequence $K_i$ of finite index subgroups.
\end{theorem}
\begin{proof}
Fix a  Folner sequence $F_i$ of universal transversals in the group $G$ as provided by   Corollary \ref{cor:residually finite amenable has Foler sequence of finite to one transverals}. Up to passing to a subsequence we may assume that the sequence $F_i$ is tempered.

Let $\mu$ be an ergodic and co-sofic invariant random subgroup of $G$. There exists a sequence of finite index subgroups $K_i$ so that $T_i * K_i \to \mu$ for any sequence finite-to-one transversals $T_i$, see Propositions  \ref{prop:if ergodic weak-* limit then is weak* limit of ergodic} and
	\ref{prop:finite index subgroups and transversals}. Up to \enquote{repeating} each subgroup $K_i$ finitely many times, we may assume that $F_i$ is indeed a finite-to-one transversal to $K_i$ for all $i \in \NN$. In particular $F_i * K_i \to \mu$  with respect to our fixed universal Folner sequence of transversals. 

 The ergodic theorem for amenable groups (Theorem  \ref{thm:pointwise ergodic for amenable}) implies that $\mu$-almost every subgroup $H$ satisfies $F_i * H \to \mu$. We deduce that $\mu$-almost every subgroup $H$ admits a Weiss approximation of the form $(K_i, F_i)$.
\end{proof}

Since the invariant random subgroup $\mu$ appearing in  Theorem \ref{thm:converse on cosofic for residually finite} is   ergodic   the same sequence  $(K_i,F_i)$ is a Weiss approximation to $\mu$-almost every subgroup.

\section{Group extensions and Goursat triples}
\label{sec:Goursat}

Let $G$ be a   discrete group. Assume   that $G$ is an extension of the group $Q$ by the group $N$. This means that  there is a short exact sequence
$$ 1 \to N \to G \to Q \to 1.$$
It will be useful to fix an arbitrary lift  $\widehat{q} \in G$ for every element $q \in Q$. This means that $\widehat{q} N = q \in Q$. It particular if $I \subset Q$ is any subset we will use the notation
$$ \widehat{I} = \{\widehat{q} \: : \: q \in I \}.$$

\subsection*{Goursat's triplets}
Given a subgroup $H$ of $G$ consider the pair of subgroups $$N_H = H \cap N \le N \quad \text{and} \quad  Q_H = HN/N  \le Q.$$
Let $\alpha_H$ denote the natural group isomorphism
$$ \alpha_H : Q_H \to H / N_H.  $$
For every element   $q  \in Q_H$ we will regard the image $\alpha_H(q )$ as a subset of the group  $G$. Namely $\alpha_H(q ) \subset H \subset G$ is a coset of the subgroup $N_H$.  Indeed
$$ \alpha_H(q ) = H \cap \widehat{q}N \quad \forall q  \in Q_H.$$

\begin{prop}[Goursat lemma for group extensions]
	\label{prop:description of subgroups of direct products}
	There is a bijective correspondence between subgroups $H$ of $G$  and triples $[Q_H, N_H,  \alpha_H]$ where 
	\begin{itemize}
		\item $Q_H$ is a subgroup of $Q$, 
		\item $N_H$ is a subgroup of $N$, and
		\item  $\alpha_H : Q_H \to \mathrm{N}_G(N_H)/N_H$ is a homomorphism satisfying   $\alpha_H(q )N_H \subset \widehat{q} N$ for every element $q  \in Q_H$.
	\end{itemize}
	\end{prop}
		We will denote   $[H] = \left[Q_H, N_H, \alpha_H\right]$ and call this   the \emph{Goursat triplet} associated to the subgroup $H$. The subgroup $H$ is uniquely determined by its Goursat triplet. Every homomorphism $\alpha_H$ participating in a Goursat triplet is necessarily injective.

	\begin{proof}[Proof of Proposition 	\ref{prop:description of subgroups of direct products}]

The   discussion preceding the statement  demonstrates how to associate to every  subgroup $H\le G$ a triple $[N_H, Q_H, \alpha_H]$ as above.  
		
For the converse, assume that we are given a triple $[Q_H, N_H, \alpha_H]$ with the above properties. Consider the   subgroup $H$ of $G$ given by
$$ H = \alpha_H(Q_H)N_H.$$
In other words $H$ is the image of the homomorphism $\alpha_H$ being regarded as a subset of the group $G$. It is easy to verify that the resulting subgroup  $H$ corresponds to the given Goursat triplet.
%
		
We conclude that the   triple $[Q_H, N_H, \alpha_H] $ uniquely determines the subgroup $H$ and that this correspondence is bijective.
	\end{proof}

The above makes it clear that the index of a subgroup $H$ of $G$ is equal to
$$\left[G:H\right] = \left[N:N_H\right]\left[Q:Q_H\right]$$
 with the understanding that $\infty \cdot n = \infty \cdot \infty = \infty$.


\begin{prop}  
\label{prop:triple notation for conjugates}
Let $H$ be a subgroup of $G$ and $g  \in G$ be an element with $q = gN \in Q$.
  Then the Goursat triplet of the conjugate subgroup $H^g$ is  
  $$[H^g]  = [N_H^g, Q_H^q, \alpha_H^g]$$ 
where the isomorphism $\alpha_H^g : Q_H^q \to H^g/N_H^g$ is given by 
$$ \alpha_H^g(r^q) N_{H^g} = (\alpha_H(r)N_H)^g$$
for every $r \in Q_H$.
\end{prop}
\begin{proof}
We leave the straightforward verification to the  reader.
\end{proof}


\subsection*{Transversals   in group extensions}

Recall that $\widehat{q} \in G$ is some fixed and arbitrary  lift for each element $q \in Q$.

\begin{prop}
	\label{prop:product of finite-to-one is finite-to-one transversal}
Let $I$ and $P$    be left transversals, respectively, for the subgroups $R \le    Q$ and $L  \le N$. 
Then 
 $ F = \widehat{I} P$ is a left transversal for any subgroup $H$ of $G$ with $Q_H = R$ and $N_H = L$.
\end{prop}
\begin{proof}
Let $H \le G$ be any subgroup with Goursat triplet 
$ [H] = [R, L, \alpha]$ 	where $\alpha : R \to \mathrm{N}_G(L)/L$ is any   homomorphism with $\alpha(r) L \subset \widehat{r} N$ for every  element $r \in R$. We claim that for every element $g \in G$ there is a unique element $h \in H$ such that $gh\in F$.
	
Consider some fixed element $g \in G$ with $gN = q \in Q$.  Let $r \in R$ be the unique element so that $qr = s \in I$ and $h_1 \in \alpha_H(r) \subset H$ be an arbitrary element. As $r=h_1N$ and $s = \widehat{s}N$ we have 
$$g h_1 =  \widehat{s} n$$
 for some element   $n \in N$. There a unique element  $h_2 \in L $ such that $n h_2 = p \in P$. Denote $h = h_1 h_2 \in H$ so that
$$g h = g h_1 h_2 =  \widehat{s} nh_2 = \widehat{s} p \in F.$$

The uniqueness of the element $h \in H$ satisfying the above condition   follows from the   uniqueness   of   $r \in R$ combined with the uniqueness of the element $h_2 \in L$.
\end{proof}

\subsection*{Folner sets in   group extensions} Folner sets for $G$ can be constructed by combing Folner sets for $Q$ and for $N$. Recall that following result of Weiss from \cite[\S2]{weiss2001monotilable}.




\begin{prop}[\cite{weiss2001monotilable}]
	\label{prop:adapted and uniformly Folner gives product Folner}
Let $I_i$ be a Folner sequence in $Q$ and $P_i$ be a Folner sequence in $N$. Then there is a subsequence $k_i$ such that 
 $F_i = \widehat{I_i} P_{k_i}$
is a Folner sequence in $G$.
\end{prop}

Unfortunately Proposition \ref{prop:adapted and uniformly Folner gives product Folner} will not suffice for our needs, as we will usually not be able to pass to a subsequence of the Folner sequence $P_i$ in the group $N$. We will have to rely on the more precise but specialized Lemma  \ref{lem:F_i is Folner}.

\section{Subgroups and invariant random subgroups of   metabelian groups}
\label{sec:subs and irs of metabelian}

Let $G$ be a finitely generated metabelian group. Therefore the group $G$ is the intermediate term in a   short exact sequence
$$ 1 \to N \to G \to Q \to 1$$
such that    the normal subgroup $N$ as well as the finitely generated quotient subgroup $Q$ are abelian.   The conjugation action of $G$ on $N$ factors through the quotient $Q$. Given an element $q \in Q$ we will frequently use the notation $n \mapsto n^q \in N$. The abelian group $N$ can be regarded as a finitely generated $\ZZ[Q]$-module with respect  to this conjugation action of $Q$.
	
	\vspace{7.5pt}
	
	\emph{From now on, when studying metabelian groups we will use multiplicative notation for the quotient group $Q$ and additive notation for the normal subgroup $N$.}


\subsection*{Residual finiteness in metabelian groups} In our work we crucially rely on the following classical theorem  \cite{hall1959finiteness}.

\begin{theorem}[Hall]
\label{thm:Hall}
Finitely generated metabelian groups are residually finite.
\end{theorem}
Hall's theorem allows us to find \enquote{Chabauty approximations from above} inside the subgroup $N$, in the following sense.
\begin{prop}
\label{prop:finite index normal subgroups in a metabelian group}
Let $H$ be a   subgroup of $G$ with Goursat triplet
$$ [H] = [Q_H, N_H, \alpha_H].$$ 
Assume that the metabelian subgroup $HN$ is finitely generated. Then for every finite subset  $T \subset N$   there is  a finite index subgroup $M \le N$ with
$$ M \cap T = N_H \cap T, \quad N_H \le M  \le N, \quad \text{and} \quad  H  \le \mathrm{N}_G(M ).$$
\end{prop}
 \begin{proof}
The subgroup $N_H$ is normalized by the subgroup $H$ as well as by the abelian group $N$. Therefore $N_H$ is a normal subgroup of $HN$. Let $\pi_H$ denote the quotient homomorphism $\pi_H:HN \to HN/N_H$.
	
The finitely generated  metabelian group $HN/N_H $ is residually finite by Hall's theorem.  So there is a finite index normal subgroup $\overline{M} \nrm HN/N_H $ such that $\pi_H(n) \notin \overline{M}$ for every element $n \in T \setminus N_H$.

 Consider the subgroup  $M = \pi_H^{-1}(\overline{M} ) \cap N$.   Since $\pi_H^{-1}(\overline{M} )$ has finite index in $HN$, the subgroup $M $ has finite index in $N$. It is clear that $N_H = \ker \pi_H \le M$. The condition $\overline{M} \cap \pi_H(T) = N_H$ implies that $M \cap T = N_H \cap T$. Finally, since $\overline{M} $ is normal in $HN/N_H$ it follows that $M $ is normalized by $H $.
\end{proof}

\subsection*{Invariant random subgroups of metabelian groups} Let $\pi$ be the natural quotient map from $\Sub{G}$ to $\Sub{Q}$. The map $\pi$ is Borel and $G$-equivariant according to Proposition \ref{prop:pushing forward invariant random subgroups}. Let $\rho$ be the restriction map from $\Sub{G}$ to $\Sub{N}$. The map $\rho$ is continuous and $N$-equivariant. 
 In light of the Goursat triplet description established in Proposition 	\ref{prop:description of subgroups of direct products} the fiber of the product map
$$ \pi \times \rho : \Sub{G} \to \Sub{Q} \times \Sub{N} $$
over a given pair of subgroups $(R,L)$ with $R \le Q$ and $L \le N$ consists of all   possible   homomorphisms $\alpha$ of the form
$$ \alpha : R \to \mathrm{N}_G(L)/L$$
where $\alpha(r) L \subset \widehat{r} N$ for all elements $r \in R$.

Recall that $Q$ is a finitely generated abelian group. In particular $Q$ is Noetherian in the sense that every subgroup of $Q$ is finitely generated. Therefore there at most countably many homomorphisms whose domain is a subgroup of $Q$. By the previous paragraph, the fibers of the product map $\pi \times \rho$ are  countable as well. 
Let us apply these observations to the study of invariant random subgroups of $G$.

\begin{prop}
\label{prop:an element of an invariant random subgroup has finite orbit for N}
Let $\mu$ an   invariant random subgroup of the finitely generated metabelian group $G$. Then $\mu$-almost every subgroup $H$ satisfies $[N:\mathrm{N}_N(H)] < \infty$.
\end{prop}

\begin{proof}
The ergodic decomposition allows us to assume without loss of generality that $\mu$ is ergodic. In particular $\pi_* \mu$ is an ergodic invariant random subgroup of the abelian group $Q$ and as such must be atomic. Therefore there is a subgroup $Q_\mu$ of $Q$ such that $\mu$-almost every subgroup $H$ satisfies $\pi(H) = Q_\mu$.

Consider the push-forward $\nu = \rho_* \mu \in \IRS{N}$ where   $\rho : \Sub{G} \to \Sub{N}$ is the restriction map as above. By the discussion preceding this proposition, the fibers of the map $\rho$ are $\mu$-almost surely countable. We would like  to conclude that the $N$-orbit by conjugation of $\mu$-almost every subgroup $H$ of $G$ is finite.

 With the above goal in mind, consider the disintegration of the probability measure  $\mu$ along the Borel map $\rho$. This is  a $\nu$-measurable map 
$$\lambda : \Sub{N} \to \Probs{G}, \quad \lambda : L \mapsto \lambda_L$$ 
with the following properties
\begin{itemize}
\item $\lambda_L(\rho^{-1}(L)) = 1$ for $\nu$-almost every subgroup $L \in \Sub{N}$, and
\item $\mu = \int_{\Sub{N}} \lambda_L \; \mathrm{d} \nu(L)$.
\end{itemize}
This means that $\lambda_L$ is a probability measure giving total measure one to the countable fiber $p^{-1}(L)$ for $\nu$-almost every subgroup $L$  of $ N$.  

 The conjugation action of $N$ preserves the fiber $\rho^{-1}(L)$ as well as the probability measure $\lambda_L$ for $\nu$-almost every subgroup $L$, see Proposition  \ref{prop:triple notation for conjugates}. Recall that a probability measure preserving action on a countable set has finite orbits.   Therefore the $N$-orbit  of $\lambda_L$-almost every point of the fiber    $\rho^{-1}(L)$ is finite $\nu$-almost surely. 
 
 The conclusion follows by relying on the integral formula $\mu = \int \lambda_L \mathrm{d} \nu$ to verify that the condition  $[N:\mathrm{N}_N(H)] < \infty$ is satisfied by $\mu$-almost every subgroup $H \le G$.
\end{proof}

\subsection*{On subgroups with finite orbit for $N$-conjugation}
Let $H$ be a subgroup of the metabelian group  $G$. We now  discuss  certain consequences of the algebraic information    $[N:\mathrm{N}_N(H)] < \infty$ to be used in \S\ref{sec:controlled approximations} below.



\begin{prop}
\label{prop:finite orbit implies finitely many elements}
 If    $[N:\mathrm{N}_N(H)] < \infty$ then  the subset $\left[h,N\right] N_H$  of the coset space $N/N_H$ is finite 
  for every $h \in H$. 
\end{prop}
\begin{proof}
Let $h \in H$ be fixed. Consider the map  $f_h$ given by
$$ f_h :   N \to N, \quad f_h(n) = \left[h,n\right].$$ 
We claim that the map $f_h$ is a group homomorphism. Indeed, given any pair of elements $n,m \in N$ it follows from Equation (\ref{eq:standard commutator identities}) that 
$$ f_h(n + m) = \left[h,n+m\right] + N_H = \left[h,n\right] + \left[h,m\right]^n + N_H = f_h(n)  + f_h(m).$$
Observe that  $f_h\left(\mathrm{N}_N(H)\right)  \subset N_H$. The assumption   $[N:\mathrm{N}_N(H)] < \infty$ implies that $f_h(N)= [h,N]$ is finite in the coset space $N/N_H$, as required.
\end{proof}

The following Corollary \ref{cor:technical lemma on Psi_q} is a more sophisticated  variant of Proposition \ref{prop:finite orbit implies finitely many elements} where we take into account an ascending sequence $M_i$ of subgroups of $N$.  Recall that a sequence of subsets  or subgroups  of $G$ is exhausting if it converges to $G$ in the Chabauty topology, see Definition \ref{def:exhausting sequence}. 
\begin{cor}
	\label{cor:technical lemma on Psi_q}
Let $M_i$ be an exhausting sequence of subgroups of $N$ such that $H \le \mathrm{N}_G(M_i)$ for all $ i\in \NN$. If $\left[N:\mathrm{N}_N(H)\right] < \infty $ then   for every   element $h \in H$ there is a finite subset $\Phi_h \subset N$ such that
$$ \left[h,M_i\right]  \subset \Phi_h + (H \cap M_i)$$ 
for 
   all $i \in \NN$ sufficiently large. 
\end{cor}
\begin{proof}
Fix an element $h \in H$.   Consider the map $f_h: N \to N$ given by
$f_h(n) = \left[h,n\right]$. It is a group homomorphism, as verified in the proof of Proposition \ref{prop:finite orbit implies finitely many elements}. It follows that there is a finite subset $\Phi_h \subset N$ such that 
$$f_h(N) = \Phi_h + N_H \subset N.$$

The family $M_i$ exhausts $N$ and we may assume that $i$ is sufficiently large so that $\Phi_h \subset M_i$. Since $h \in H \le \mathrm{N}_G(M_i)$ we have $f_h(M_i) \subset M_i$. 

Observe that every element $x \in f_h(M_i)$ satisfies $x \in M_i$ as well as $x = f + n$ for a pair of elements $f \in \Phi_h \subset M_i$ and $n \in N_H$. Therefore $n = x-f \in M_i$ and hence $$x = f + n \in \Phi_h + (N_H \cap M_i).$$
 We conclude that
$$ f_h(M_i) \subset  \Phi_h + (H \cap M_i) $$
as required.
\end{proof}

While it is possible to prove Corollary \ref{cor:technical lemma on Psi_q} directly and then deduce Proposition \ref{prop:finite orbit implies finitely many elements} as a special case, the current outline seems to be more transparent.

\subsection*{Formulas for conjugates in   metabelian groups}

 The following is an elementary commutator computation     that can be used   to decide which elements in a given conjugacy class belong to a given subgroup.

\begin{prop}
	\label{prop:condition on conjugate to be in H}
Let $H$ be a subgroup of $G$. Consider a pair of elements $g,f \in G$ with
	$$ g = \widehat{q} n, \quad f = \widehat{r} m, \quad q,r \in Q, \quad \text{and} \quad n,m \in N.$$
	Then $g^f \in H$ if and only if both $q \in Q_H$ and
	$$  \left[g,f\right] + (n - a_q) \in N_H $$
	where $a_q\in N$ is any element satisfying 
	$$ \alpha_H(q)N_H = H \cap \widehat{q}N = \widehat{q}( a_q + N_H).$$
\end{prop}
 
 \begin{proof}
 We may use Equation (\ref{eq:standard commutator identities}) to express the  conjugate $g^f$ as
$$
	 g^f = g\left[g,f\right] = \widehat{q} (n + \left[g,f\right]).$$
	In particular $g^f \in H$ if and only if $q\in Q_H$ as well as
	$$ n + \left[g,f\right]  \in  a_q + N_H.$$
	The result follows.
\end{proof}

%

\section{Consistent homomorphisms}


Let $A$ and $B$ be discrete groups. Let $C_i$ be a family of normal subgroups of $B$. Consider a family of homomorphisms
$ f_i : A \to B/C_i$.
\begin{definition}
	\label{def:consistent}
	The family $f_i$ is \emph{consistent}  if for every  $a \in A$ there exists an element $b_a \in B$  with
	$f_i(a) = b_a C_i$
	for all $i$.
\end{definition}

Chabauty convergence can be used to get a consistent family of maps.

\begin{lemma}
	\label{lem:on kernel of a sequence of homomorphisms and Chabauty}
	Assume that $A$ is finitely presented. Let $f : A \to B/C$ be a homomorphism where $C$ is some  normal subgroup of $B$.  If the normal subgroups $C_i$ Chabauty converge to $C$ then   there exists a family of homomorphisms $f_i : A \to B/C_i$
	consistent with the homomorphism $f$ for all $i \in \NN$ sufficiently large.
\end{lemma}

\begin{proof}
	
	Let   $\Sigma$ be a finite generating set and $R$ be a finite set of defining relations for the group $A$.  Let $F_\Sigma$ be the free group on the set $\Sigma$ with quotient map $p : F_\Sigma \to A$ so that $N = \ker p$ is normally generated by  $R$ regarded as a subset of $F_\Sigma$.

Consider the quotient maps
$$q : B \to B/C \quad \text{and} \quad q_i : B \to B/C_i \quad \forall i \in \NN.$$ Choose an  element $b_\sigma \in B$ for each generator $\sigma \in \Sigma$ such that 
	$$b_\sigma C = q(b_\sigma) = f \circ p(\sigma) = f(\sigma N).$$ 
	Let $\overline{f}: F_\Sigma \to B$ be the homomorphism defined on each generator $\sigma \in \Sigma$ of the free group $F_\Sigma$ by $\overline{f}(\sigma) = b_\sigma$.  Note that $f \circ p = q \circ \overline{f} $. Therefore $\overline{f}(R) \subset C$. The sequence $C_i$ Chabauty converges to $C$ by assumption. Let $i_0 \in \NN$ be sufficiently large index  so that $\overline{f}(R) \subset C_i$ and therefore     $\overline{f}(N) \le C_i$ for all $i > i_0$. 
	
 Let $s : A \to F_\Sigma$ be an arbitrary section to the quotient map $p$, namely  $p \circ s (a) = a$  holds true for all elements $a \in A$. For each $ i\in \NN$ consider the map $f_i : A \to B/C$ defined  by
$$f_i = q_i \circ \overline{f} \circ s : A \to B/C_i.$$

We claim that $f_i$ is a group homomorphism for all indices $i > i_0$.   Indeed, for each pair of elements $a_1,a_2 \in A$ the section    $s$ satisfies 
$ s(a)s(b)N = s(ab)N$. Therefore
\begin{equation}
\label{eq:longish ugly}
 (\overline{f} \circ s)(a) (\overline{f} \circ s)(b)C_i = (\overline{f} \circ s)(ab) C_i.\end{equation}
The claim follows by observing that Equation (\ref{eq:longish ugly})  implies  $f_i(a) f_i(b) =   f_i(ab)$.

It remains to verify that the family of homomorphisms $f_i$ is consistent with $f$. Precomposing the equation $f \circ p = q \circ \overline{f} $ with the section $s$ gives
$$f = f\circ p \circ s = q \circ \overline{f} \circ s : A \to B/C.$$
We conclude that   all elements $a \in A$ satisfy
$$f(a) = (\overline{f} \circ s)(a)C \quad \text{and} \quad f_i(a) = (\overline{f} \circ s)(a)C_i$$
which is exactly the requirement in  Definition \ref{def:consistent}.
\end{proof}

Let $G$ be a finitely generated metabelian group with normal abelian subgroup $N$ and abelian quotient $Q \cong G/N$.  Let $H$ be a subgroup of $G$ with Goursat triplet 
$ [H]  = [Q_H, N_H, \alpha_H]$.
\begin{cor}
	\label{cor:on construction of subgroups using Goursat triplets}

	Let $N_i$ be a sequence of subgroups of $N$ satisfying $H \le \mathrm{N}_G(N_i)$ and Chabauty converging to $N_H$. Then there exists a   family of   homomorphisms
	$$\alpha_i : Q_H \to \mathrm{N}_G(N_i)/N_i$$ 
	defined for all $i \in \NN$ sufficiently large and consistent with the isomorphism $\alpha_H$.
\end{cor}
The fact that the homomorphisms $\alpha_i$ are consistent with the homomorphism $\alpha_H$ automatically implies that $\alpha_i(q) N_i \subset \widehat{q}N$ for all elements $q \in Q_H$ and all $ i\in \NN$.
\begin{proof}[Proof of Corollary \ref{cor:on construction of subgroups using Goursat triplets}]
	Recall that the map $\alpha_H$ is an isomorphism   from $Q_H$ to $H /N_H$. Since $Q$ is a finitely generated abelian group, its subgroup $Q_H$ is finitely generated as well. The image of $\alpha_H$ can be regarded as being contained in the larger group $HN/N_H$.   We may therefore apply Lemma \ref{lem:on kernel of a sequence of homomorphisms and Chabauty}    with respect to the groups
	$$ A = Q_H, \quad B = HN, \quad C = N_H, \quad \text{and} \quad C_i = N_i $$
	to obtain family of homomorphisms $\alpha_i : Q_H \to HN/N_i$ consistent with the homomorphism $\alpha_H$ for all $i\in \NN$ sufficiently large. As $N_i \nrm N$ and $H  \le \mathrm{N}_G(N_i)$ it follows that the image of each $\alpha_i$ lies in $\mathrm{N}_G(N_i)/N_i$. 
\end{proof}

%
%
%
%
%

\section{Controlled approximations}
\label{sec:controlled approximations}

Let $G$ be a finitely generated metabelian group with a normal abelian subgroup $N$ and abelian quotient group $Q \cong G/N$. 

Let   $H \le G$ be any subgroup. We introduce an auxiliary notion of    controlled approximations for $H$. We also define an accompanying notion of a sequence of finite-to-one transversals being adapted to a given controlled approximation.

 Under certain favorable algebraic circumstances, that always apply if $H$ lies in the support of an invariant random subgroup of $G$,  a controlled approximation with its adapted sequence of transversals gives rise to a Weiss approximation, see Theorem \ref{thm:a group with a controlled approximation is co-sofic}.

\subsection*{Controlled approximations}

Let $\left[H\right] = \left[Q_H, N_H, \alpha_H\right]$ be the Goursat triplet of the subgroup $H$ of the finitely generated metabelian group $G$. 

 We briefly discuss the underlying idea of a controlled approximation prior to stating the formal definition. Roughly speaking, this is a sequence of finite index subgroups $K_i$ of $G$ with Goursat triplets $\left[K_i\right] = \left[Q_i, N_i, \alpha_i\right]$.

 We require that the sequences $Q_i$ and $N_i$ Chabauty converge to the subgroups $Q_H$ and $N_H$, respectively, and that the homomorphisms $\alpha_i$ are consistent with $\alpha_H$ in the sense of Definition \ref{def:consistent}. 

While we may assume without loss of generality that the $Q_i$'s approximate $Q_H$ \enquote{from above}  in the sense that $Q_H \le Q_i$ for all $i \in \NN$, this is not the case for the approximation of $N_H$ by the $N_i$'s.  The main point is that   the Chabauty convergence of the $N_i$'s to $N_H$ is    \enquote{controlled} by a sequence of subgroups $M_i \le N$. The restriction of the approximation $N_i$ to each subgroup $M_i$ is \enquote{from above}, namely $N_H \cap M_i \le N_i \cap M_i$. 

Finally, the \enquote{speed} of the Chabauty convergence of the $N_i$'s towards the subgroup $N_H$ is  \enquote{controlled} by a sequence of subsets $T_i \subset N$.



\begin{definition}
\label{def:controlled approximation}
	A \emph{controlled approximation} for $H$ is a sequence  $(K_i, M_i, T_i)$ where 
\begin{itemize}
\item $K_i \le G $ is a finite index subgroup     with Goursat triplet  
	$ \left[K_i\right] = \left[Q_i, N_i, \alpha_i \right]$,
\item $M_i \le N$ is a subgroup with $H \le \mathrm{N}_G(M_i)$, and
\item   $T_i \subset M_i$ is a subset 
\end{itemize}	
 such that
\begin{enumerate}
\item 
\label{it:Q_i converge}
the subgroups $Q_i$ Chabauty converge to $Q_H$,
\item 
\label{it:restrictions consistent}
the  maps $  \alpha_{i|Q_H}$ are consistent with $\alpha_H$, 
\item 
\label{it:sequence exhausts}
the sequence $T_i$ exhausts the subgroup $N$, 
\item 
\label{it:subgroups satisfy}
the subgroups $N_i$ satisfy
		\begin{equation}
		\label{eq:controlled} 
		\forall i \in \NN \quad N_H \cap T_i = N_i \cap T_ i \quad   \text{and} \quad N_H \cap M_i \le N_i \cap M_i.
		\end{equation}
	\end{enumerate}
\end{definition}

If the sequence $K_i$ is a controlled approximation for the subgroup $H$ then  by Items (\ref{it:sequence exhausts}) and (\ref{it:subgroups satisfy}) the sequence $N_i = K_i \cap N$   Chabauty converges to the subgroup $N_H$. Since $Q_H$ is finitely generated we have $Q_H \le Q_i$ for all $i \in \NN$ sufficiently large and the restrictions $\alpha_{i|Q_H}$ are well defined.

We will frequently abuse   notation and simply refer   to the sequence $K_i$ of finite index subgroups as a controlled approximation, without an explicit mention of  the remaining data (i.e.  the subgroups $M_i$ and the subsets $T_i$).


\subsection*{Adapted sequences of transversals}
The following two definitions introduce a companion notion to that of controlled approximations.

	\begin{definition}
		\label{def:adapted sequence}
		The sequence $A_i$ of finite subsets of $G$ is \emph{adapted} to the sequence $B_i$ of finite subsets of the group $G$  if for any finite subset $\Phi \subset N$ and for every element $g \in G$ 
		\begin{equation}
		\label{eq:adapted def}
		\frac{|\{b \in B_i \: : \: \left[g,b\right] + \Phi \subset A_i \}|}{|B_i|} \to 1.
		\end{equation}
	\end{definition}
	
\begin{definition}
\label{def:adapted transversals}
 A   sequence of finite-to-one transversals $F_i \subset G$ of the   subgroups $\mathrm{N}_G(K_i)$  is \emph{adapted} to the controlled approximation $(K_i,M_i,T_i)$  if
\begin{itemize}
	\item  $F_i = \widehat{I}_i P_i$ with $I_i \subset Q$ and $P_i \subset N$ for all $i \in \NN$,
		\item the sequences $I_i$ and $P_i$ exhaust the groups  $Q$ and $N$ respectively,  and
	\item the sequence of subsets $T_i$ is adapted to the sequence $\widehat{I}_i$.
\end{itemize} 
\end{definition}

In our application it will so happen that $T_i = P_i$. This coincidence is not mandatory for the proof of Theorem \ref{thm:a group with a controlled approximation is co-sofic} given below.


\subsection*{From controlled approximations to Weiss approximations} A controlled approximation is in some sense an algebraic notion,  while a Weiss approximation is statistical in nature. For metabelian groups it is nevertheless possible to go from one notion to the other.

\begin{lemma}
\label{lem:application of controlled approximation}
If $(K_i, M_i,T_i)$ is a controlled approximation of the subgroup $H$ then
$$T_i + (N_H \cap M_i)\subset   N \setminus (N_H \triangle N_i)$$
where $N_i = K_i \cap N$ for all $i\in\NN$.
\end{lemma}
\begin{proof}
Consider an element $x \in N$ equal to
  $ x = t + n$ where $  t \in T_i $ and  $ n \in N_H \cap M_i$.
  From Equation (\ref{eq:controlled})  of Definition \ref{def:controlled approximation}  we have that $  t \notin   N_H \triangle N_i $ and $n \in N_H \cap N_i$ for all $i\in \NN$.   There are two cases to consider. 
  \begin{itemize}
  \item If $t \in N_H \cap N_i$ then $ x = t + n \in N_H \cap N_i$.
  \item If $t \notin  N_H \cup N_i $ then $ x \notin N_H \cup N_i$ since $ t = x- n$ and $n \in N_H \cap N_i$.
\end{itemize}
In either  case    $x \notin  N_H \triangle N_i$.
\end{proof}

\begin{theorem}
\label{thm:a group with a controlled approximation is co-sofic}
  Let $(K_i,M_i,T_i)$ be a controlled approximation to the subgroup $H$ and $F_i$ be an adapted sequence of finite-to-one transversals of $\mathrm{N}_G(K_i)$ in   $G$. If $\left[N:\mathrm{N}_N(H)\right] < \infty$, then  the sequence $(K_i,  F_i)$ is a Weiss approximation for $H$.
\end{theorem}


\begin{proof} 


We wish to prove that  the sequence $(K_i, F_i)$ is indeed a Weiss approximation for the subgroup $H$.  With this goal in mind    fix an element $g = \widehat{q}n \in G$ for some pair of  elements $q \in Q$ and $n \in N$.
Relying  on the explicit condition given in Proposition \ref{prop:a numberical condition for a group to be co-sofic} we need to verify that
\begin{equation}
\label{eq:desired limit}
 p_i(g) = \frac{|\{f \in F_i \: : \: g^f \in K_i \triangle H \}| }{|F_i|} \xrightarrow{i \to \infty} 0.
\end{equation}	
There are two separate cases to deal with,  depending on whether the element $q$ belongs to the subgroup $Q_H$ or not. 
	
	\vspace{5pt}
	
The first case to deal with is where   $q \notin Q_H$. Then by Item (\ref{it:Q_i converge}) of Definition \ref{def:controlled approximation} we deduce that $q \notin Q_i$ as well for all $i \in \NN$ sufficiently large. As the group $Q$ is abelian, this means that $g^f \notin  K_i \triangle H $ for all elements $f \in G$ whatsoever and certainly for all $f \in F_i$.   In particular    $p_i(g) = 0$ for all $i\in \NN$ sufficiently large.

	\vspace{5pt}
		
	The second and more interesting case is where $q \in Q_H  $. By Item (\ref{it:restrictions consistent}) of Definition \ref{def:controlled approximation}, the restrictions   $\alpha_{i|Q_H}$ are  consistent with the homomorphism $\alpha_H$. Therefore there is an element $a_q \in N$ such that	 for all $i \in \NN$ we have
	$$\alpha_H(q) = \widehat{q}( a_q+ N_H) \quad \text{and} \quad \alpha_i(q) = \widehat{q} (a_q +N_i). $$
	
	To estimate the quantity $p_i(g)$ consider a given element $f \in F_i = \widehat{I}_i P_i$ where
		$$ f = \widehat{r}m, \quad r \in I_i, \quad \text{and} \quad m   \in P_i.$$ 
Recall that Proposition \ref{prop:condition on conjugate to be in H}  
gives an explicit criterion for the conjugate $g^f$ to be contained in either subgroup $H$ or $K_i$ of $G$. As we assumed that $q \in Q_H$ and $i$ is sufficiently large so that $q \in Q_i$, this criterion implies that
\begin{equation}
\label{eq:iff}
g^f \in H \triangle K_i  \quad \text{if and only if} \quad \left[g,f\right] + \delta \in  N_H \triangle N_i.
\end{equation}
where $\delta = n - a_q \in N$ depends only on $g$. 

Since $q \in Q_H$ there some element $h \in H$ such that $hN = gN$. Equation (\ref{eq:standard commutator identities}) implies that $\left[g,m\right] = \left[h,m\right]$. Therefore Corollary \ref{cor:technical lemma on Psi_q}  supplies a finite subset $\Phi_g \subset N$ so that
\begin{equation}
\label{eq:x(g,f) 2nd}
\left[g,f\right] = \left[g,m\right] + \left[g,\widehat{r}\right] = \left[h,m\right] + \left[g,\widehat{r}\right]
\in \left[g,\widehat{r}\right] +   \Phi_g +   (N_H \cap M_i).
\end{equation}
On the other hand,  Lemma \ref{lem:application of controlled approximation} tells us that
\begin{equation}
\label{eq:agree}
T_i + (N_H \cap M_i)\subset   N \setminus (N_H \triangle N_i).
\end{equation}
By putting together     equations $(\ref{eq:iff})$, $(\ref{eq:x(g,f) 2nd})$ and $(\ref{eq:agree})$  we infer that
\begin{equation}
\label{eq:final}
\left[g,\widehat{r}\right] +  (\Phi_g + \delta) \subset T_i \quad \text{implies} \quad g^f \in G \setminus (H \triangle K_i).
\end{equation}

To conclude the proof of the   case where $q \in Q_H$, we are required to show that $p_i(g) \xrightarrow{i\to\infty} 0$. Relying on Equation (\ref{eq:final}) this would follow from 
\begin{equation}
\label{eq:special adapted}
 \frac{|\{r \in I_i \: : \: \left[g,\widehat{r}\right] +  (\Phi_g + \delta) \subset T_i \} | }{|I_i|} \xrightarrow{i\to\infty} 1.
 \end{equation}
  By assumption, the sequence $F_i$ of finite-to-one transversals is adapted to the controlled approximation $(K_i,M_i,T_i)$ in the sense of Definition \ref{def:adapted transversals}. This means that the sequence  $T_i$ is adapted to the sequence of lifts $\widehat{I}_i$ in the sense of Definition \ref{def:adapted sequence}.
  Finally, observe that Equation (\ref{eq:special adapted}) follows as a special case of Equation (\ref{eq:adapted def}).
 \end{proof}


\section{Permutational metabelian groups}
\label{sec:permutational metabelian groups}


Let $Q$ be a finitely generated abelian group. A \emph{$Q$-set} is a set admitting an action of the group $Q$. Let $X$ be a fixed $Q$-set. 

\subsection*{Permutational modules}

Let $B$ be a finitely generated abelian group. 
The \emph{permutational   module} corresponding to the $Q$-set $X$ and with base group $B$ is the abelian group
\begin{equation} B^X = \bigoplus_{x \in X} B^x
\end{equation}
where for every $x \in X$ the direct summand $B^x$ is an isomorphic copy of $B$. The group $Q$ is acting on $N$ by group automorphisms by permuting coordinates. A permutational module is a module in the standard sense over the group ring $\ZZ\left[Q\right]$. The $\ZZ\left[Q\right]$-module $B^X$ is finitely generated if and only if the $Q$-set $X$ admits finitely many $Q$-orbits.

For every subset $Y \subset X$ we will  denote
\begin{equation}
 B^Y = \bigoplus_{x \in Y} B^x. 
 \end{equation}
The group $B^Y$ is a direct summand of the permutational module $B^X$ with its additive group structure.  
Similarly, given a subset $C$ of $B$ and a subset $Y$ of $X$ we will denote
\begin{equation}
\label{eq:C^Y def}
 C^Y = \{ b \in B^Y \: : \: b^x \in C \quad \forall x \in Y\}.
 \end{equation}

\begin{remark}
\label{remark:direct product vs sum}
We caution the reader that, opposite to what is customary,  we   use the notation $B^X$ for a direct sum  rather than a direct product.
\end{remark}


\begin{prop}
\label{prop:kernel acts trivially}
Let $K\nrm Q$ be the kernel of the $Q$-action on $X$. Then the subgroup $K$ acts trivially on the permutational module $B^X$.
\end{prop}
\begin{proof}
Immediate from the definitions.
\end{proof}

\begin{prop}
\label{prop:BY for Y invariant is normalized}
Let $R$ be a subgroup of $Q$. If the subset $Y \subset X$ is $R$-invariant then the subgroup $B^Y $ of the permutational module $B^X$ is $R$-invariant.
\end{prop}
The above means that $B^Y$ is a $\ZZ[R]$-submodule of the module $B^X$.

\begin{proof}[Proof of Proposition \ref{prop:BY for Y invariant is normalized}]
Note that an element $b \in B^X$ lies in $B^Y$ if and only if $b_i = 0$ for all $i\in X \setminus Y$. If the subset $Y \subset X$ is $R$-invariant then this condition is   $R$-invariant as well.
\end{proof}
 
 A \emph{$Q$-factor} of the $Q$-set $X$ is a $Q$-set $\overline{X}$ admitting a $Q$-equivariant surjective map $f : X \to \overline{X}$. 
 A \emph{transversal} to a $Q$-factor map $f$ is a subset $Y\subset X$ such that the restriction of $f$ to $Y$ is a bijection onto $\overline{X}$.

\begin{lemma}
\label{lem:if base is abelian there is a quotient of wreath products}
 Let $f : X \to \overline{X}$ be a $Q$-factor map with transversal $Y \subset X$.  
\begin{enumerate}
	\item \label{it:surjection}
	There is a surjective $\ZZ\left[Q\right]$-module homomorphism
	$ \pi_f : B^X \to B^{\overline{X}}$.
	
	\item  \label{it:restriction}
	The restriction of $\pi_f$ to the subgroup $B^Y$ is an isomorphism onto $B^{\overline{X}}$.
	\item  \label{it:intersection}
	If $H$ is a subgroup of $B^Y$ then $H + (\ker \pi_f \cap B^Y) = H$. 
	\item  \label{it:transveral}
	 Let $L$ be a subgroup of $B^{\overline{X}}$. Then $T \subset B^Y$ is a finite-to-one transversal for $\pi_f^{-1}(L)$ in $B^X$ if and only if $\pi_f(T)$ is a finite-to-one transversal for $L$ in $B^{\overline{X}}$.
\end{enumerate}

\end{lemma}
This elementary lemma crucially relies on the base group $B$ being abelian. Note the great similarity to Gr{\"u}nbergs  \cite[Lemma 3.2]{grunberg1957residual}.
\begin{proof}
The surjective homomorphism $\pi_f$ is given by the natural isomorphisms
$$ \pi_{|B_x}: B_x \xrightarrow{\cong} B_{f(x)}$$
on each coordinate.
 Since the base group $B$ is abelian $\pi_f$ is well-defined. Moreover the restriction of $\pi_f$ to $B^Y$ is clearly an isomorphism onto $B^{\overline{X}}$. This shows   $(\ref{it:surjection})$ and $(\ref{it:restriction})$. 
 Items $(\ref{it:intersection})$ and $(\ref{it:transveral})$ immediately follow as the subgroup $B^Y$ is a complement to the subgroup $\ker \pi_f$ in $B^X$, and relying on the isomorphism theorem for groups. 
\end{proof}

For example, if $R$ is a  (normal) subgroup of $Q$ then the space of $R$-orbits $ R \backslash X$ is a $Q$-factor of $X$.

\subsection*{Permutational wreath products}

 A \emph{permutational metabelian group} is a metabelian group $G$ admitting a short exact sequence $$ 1 \to N \to G \to Q \to 1$$
with both groups $N$ and $Q$  being abelian  and such that the action of $Q$ on $N$ is a permutational module with respect to some  $Q$-set $X$. A permutational metabelian group $G$ is finitely generated if and only if  $Q$ is finitely generated and the $Q$-set $X$ admits finitely many orbits.

A \emph{permutational  wreath product} with respect to the $Q$-set $X$ and with base group $B$ is a split permutational metabelian group, namely 
$$G = B \wr_X Q \cong Q \ltimes N $$
where $N$ is the permutational module $B^X$ with its additive group structure.
We will regard the permutational wreath product $G$ as an extension of the group $Q$ by the group $N$. 

A special case of the above class of groups  is the \emph{wreath product} $G = B \wr Q$. This is simply the permutational wreath product  $B \wr_Q Q$ where the group $Q$ is acting on itself by translations. This group is of course split as well.

The  \emph{free metabelian group} is an example of a non-split permutational metabelian group. In fact, if $ x_1, \ldots, x_d$ is an ordered generating set for the free metabelian group $\Phi$ then its derived subgroup  $  [\Phi, \Phi]$ is a free $\ZZ[x_1^{\pm1},\ldots,x_d^{\pm1}]$-module with basis given by the set of commutators $\{[x_i,x_j] \, : \, i < j \}$. This follows from the \emph{commutator collecting process} introduced by Phillip Hall in \cite{hall1934contribution}.

\subsection*{Commutator identities} 
Let $G$ be a split metabelian group. The standard commutator identity as in   Equation (\ref{eq:standard commutator identities}) implies that 
\begin{equation}
\label{eq:metabelian identities}
\left[qn,rm\right] = \left[q,rm\right]^n + \left[n,rm\right]   = \left[q,r\right]^m + \left[q, m\right]   + \left[n,r \right] ^m + \left[n,m\right] = \left[q, m\right]   - \left[r,n \right]
\end{equation}
for any choice of elements $q,r \in Q$ and $n,m \in N$.

\section{Finitely generated abelian groups}

Let $Q$ be a finitely generated abelian group. The structure theorem for finitely generated abelian groups allows us to decompose $Q$ as a direct sum
$Q = A \oplus T$ where $A$    is a free abelian group of rank $d$ for some $d \in \NN\cup \{0\}$ and $T$ is   torsion. Fix an arbitrary basis $\Sigma \subset Q$ for the free abelian group $A$.

\begin{definition}
Given an element   $q \in Q$  with
	$$ q = t \prod_{\sigma \in \Sigma}  \sigma^{n_\sigma}, \quad n_\sigma \in \ZZ, \quad t \in T$$
define the \emph{semi-norm} $ \|q\|_\Sigma$ of $q$ with respect to the basis $\Sigma$ as
$$ \|q\|_\Sigma  = \max_{\sigma \in \Sigma} | n_\sigma |.$$
\end{definition}

Note that $\|\cdot\|_\Sigma $ is a well-defined norm   on the torsion-free part  $A \le Q$ while $\|t\| = 0$ for all $t \in T$.   The definition of   $\|\cdot\|_\Sigma $ clearly depends on the choice of the particular basis $\Sigma$. 

Roughly speaking, balls with respect to   $\|\cdot\|_\Sigma$  look like \enquote{cubes} in the group $Q$. The ball $\{q \in Q \: : \: \|q\|_\Sigma \le k\}$ has  \enquote{sides} of length $2k+1$, which is always an odd number. For our purposes, we will require  \enquote{cubes} which look almost like $\|\cdot\|_\Sigma$-balls but allow for arbitrary  \enquote{sides}.

\begin{definition}
	\label{def:infty balls}
	The subset $\mathrm{B}_Q(k, \Sigma)$ of the group $Q$ is given by
	$$ \mathrm{B}_Q(k, \Sigma)= \{ t \prod_{\sigma \in \Sigma}  \sigma^{n_\sigma}  \: : \: -\ceil{k/2} < n_\sigma \le \floor{k/2} \quad \forall \sigma \in \Sigma, \; t \in T \}$$
	for every $k\in \NN$.
\end{definition}

The subset $\mathrm{B}_Q(k,\Sigma)$ coincides with a $\|\cdot\|_\Sigma$-ball with for $k$ odd, while for $k$ even it is  \enquote{off by one}. The torsion part $T$ is contained in $\mathrm{B}_Q(k,\Sigma)$ for every $ k \in \NN$.


 For every $k\in \NN$, let $Q[k]$ denote the subgroup generated by the $k$-th powers of all elements of $Q$. The subgroup $Q[k]$ has finite index and is characteristic in $Q$.
The following is essentially an immediate consequence of  the above definitions.
\begin{prop}
	\label{prop:infty balls are a transversal}
	Let $k \in \NN$. The subset $\mathrm{B}_Q( k, \Sigma)$ is a finite-to-one transversal to the subgroup $Q[k] $ for every $k \in \NN$. In fact $\mathrm{B}_A(k,\Sigma) $ is a (one-to-one) transversal to $A[k]$ where $A$ is the torsion-free part of the group $Q$.
\end{prop}

Universal sequences of transversals were introduced in Proposition \ref{prop:universal sequence} above.

\begin{cor}
\label{cor:universal sequence in fg abelian}
The sequence $B_Q(k!,\Sigma)$ is a universal  sequence of transversals in the group $Q$.
\end{cor}
\begin{proof}
This follows immediately from the observation that any finite index subgroup $R \le Q$   with $d = \left[Q:R\right]$ satisfies $Q[d] \le R$.  
\end{proof}



The subsets $\mathrm{B}_Q(k,\Sigma)$ are all centered at the identity element of the group $Q$. This behavior is captured in the following definition that will play an important role in the following section.

\begin{definition}	
	A sequence $I_i$ of finite subsets  of the group $Q$ is \emph{centered} if for any element $r \in Q$
	\begin{equation}
	\frac{|\{q \in I_i \: : \: r \in q   I_i \}|}{|I_i|} \xrightarrow{i\to\infty} 1.
	\end{equation}

\end{definition}

%
%
%
	We leave the easy verification of the following fact to the reader.
\begin{prop}
	\label{prop:infty balls are a centered Folner sequence}
	Let $k_i \in \NN$ be any sequence with $\lim k_i = \infty$. Then $\mathrm{B}_Q(k_i, \Sigma)$ is a centered exhausting Folner sequence in the group $Q$.
\end{prop}

\section{Approximations in permutational wreath products}
\label{sec:approximation in permutational}
Let $G$ be a finitely generated permutational wreath 
$$ G = B \wr_X Q$$
where $B$ and $Q$ are  finitely generated abelian    groups and $X$ is a $Q$-set. Since the group  $G$ is finitely generated,  $X$ admits finitely many $Q$-orbits, see  \S\ref{sec:permutational metabelian groups}.  We emphasize that $G$ is split. In fact
$$ G = Q \ltimes N \quad \text{where} \quad N = B^X.$$

Our current goal is the following construction.

\begin{theorem}
\label{thm:every subgroup admits a controlled approximation}
 Let $R$ be a fixed subgroup of $Q$.
Then $G$ admits a   Folner sequence 
$F_{ i}$ 
with the following property:

If $H$ is any subgroup of $G$ with
 $R = Q_H $  and 
$\left[N:\mathrm{N}_N(H)\right] < \infty$
  then $H$ has a controlled approximation $(K_i,M_i,T_i)$  such that $F_i$ is an adapted sequence of finite-to-one transversals   of $\mathrm{N}_G(K_i)$ for all $i \in \NN$ sufficiently large.
\end{theorem}
Controlled approximations and  adapted sequences of finite-to-one transversals  were introduced in Definition \ref{def:controlled approximation} and \ref{def:adapted transversals}, respectively.

\vspace{5pt}

\emph{The proof of Theorem \ref{thm:every subgroup admits a controlled approximation} will consist of several consecutive lemmas. Every   lemma will build on the previous ones.  We will allow ourselves to   use any   notations and objects introduced in this section without an explicit mention.  }

\subsection*{Chabauty approximations }
Let $R$ be a fixed subgroup of $Q$. Choose an arbitrary basis $\Sigma$ for the torsion-free part of $Q$.

  Let $r \in \NN$ denote the number of the $Q$-orbits in the $Q$-set $X$ and   $X = \coprod_{l=1}^r X_l$ be the corresponding orbit decomposition. Choose an arbitrary point $x_l \in X_l$ for every $l \in \{1,\ldots,r\}$. In particular $X_l \cong Q/S_l$ where $S_l = \mathrm{Stab}_Q(x_l) \le Q$.

\begin{lemma}
\label{lem:controlled Folner sets and transversals}
There is a monotone sequence $n_i \in \NN $   such that 
\begin{enumerate}
\item \label{it:Chabauty}
the finite index subgroups $Q_i = R +  Q[n_i]$ Chabauty converge to $R$, 
\item \label{it:direct sum}
the subgroups $Q_i$ satisfy $Q_i = R \oplus V_i$ for some subgroups $V_i \le Q$, 
\item \label{it:boxes}
the subsets $I_i = \mathrm{B}_Q(\Sigma, n_i)$ form a centered exhausting Folner sequence of finite-to-one transversals of the subgroups $Q_i$, and
\item \label{it:transversal}
the $Q$-factor maps $f_i: X \to V_i \backslash X$ admit $R$-invariant transversals $Y_i \subset X$ 
\end{enumerate}
for all $i \in \NN$.
\end{lemma}

Roughly speaking $\mathrm{B}_Q(\Sigma,n)$ is a ball of diameter $n$ based at the identity in the seminorm $\|\cdot\|_\Sigma$ with respect to the basis  $\Sigma$. See Definition \ref{def:infty balls} for details.

\begin{proof}[Proof of Lemma \ref{lem:controlled Folner sets and transversals}]
By  the structure theorem for finitely generated abelian groups there is a direct sum decomposition
$$ Q = U \oplus V$$
such that $R$ is a finite index   subgroup of $U$, the torsion subgroup $T$ is contained in $U$ and the subgroup $V$ is torsion-free. In particular, there is some $m \in \NN$   such that $  U[m] \le R  $.
The same theorem allows us to find   direct sum decompositions
$$ V = W_l \oplus W'_l $$
such that $R+S_l$ is a finite index subgroup of $U \oplus W_l$ for every $l \in \{1,\ldots,r\}$. So there are integers $m_l \in \NN$   such that $ (U \oplus W_l)[m_l] \le R+S_l$.   Take 
$$k = \mathrm{lcm}(m,m_1,\ldots,m_r).$$

Define  $n_i = ik$ and   $ Q_i = R +  Q[n_i]$ for all $i \in \NN$. Note that 
$$Q_i =  R \oplus V_i \quad \text{where} \quad V_i =  V[n_i]. $$
The subgroups   $Q_i$ Chabauty converge to the subgroup $R$. This concludes the proof of Items $(\ref{it:Chabauty})$ and $(\ref{it:direct sum})$.

Item $(\ref{it:boxes})$  concerning the subsets  $I_i = \mathrm{B}_Q(\Sigma, n_i)$ follows immediately   by combining Propositions \ref{prop:infty balls are a transversal} and \ref{prop:infty balls are a centered Folner sequence}.

It remains   to establish Item $(\ref{it:transversal})$. Note that
$$ R+ S_l +  Q[n_i] = Q_i +S_l   = (R+S_l) \oplus W'_{i,l} $$
where $W'_{i,l} =  W'_l[n_i]$ for every $ i\in\NN$ and $ l\in\{1,\ldots,r\}$. 
Note that $S_l + V_i = S_l \oplus W'_{i,l}$. Moreover $X_l = Q/S_l$ and
$$  V_i \backslash X_l = V_i \backslash Q / S_l \cong Q/(S_l + V_i) \cong Q/(S_l \oplus W'_{i,l})$$
for all $ l \in \{1,\ldots,r\}$. Consider the following subsets
$$ Y_i = \coprod_{l=1}^r Y_{i,l} \quad \text{where} \quad Y_{i,l} =  \left(U + W_l + I_i  \right) x_l. $$
Since  $R \le U$, it is clear that every subset $Y_{i,l}$ as well as the subset $Y_i$ is $R$-invariant. Moreover every subset $Y_{i,l}$ is   a transversal to the restricted $Q$-factor map
$$ X_l \cong Q/S_l \to Q/(S_l \oplus W'_{i,l}) \cong V_i \backslash X_l. $$
This concludes the proof.
\end{proof}


%
%


Denote
$Z_i = \coprod_{j=1}^r I_i x_j$. As $I_i$ exhausts $Q$, it  follows that $X = \bigcup_i Z_i$. The inclusions
 $Z_i \subset Y_i \subset X$ hold for all $i \in \NN$. The subsets $Z_i$ are always finite. The subsets $Y_i$ are infinite if and only if $R$ admits infinite orbits.

\subsection*{Adapted sequences}
The metabelian group $G$ is split. In particular, we may identify each $I_i$   with its lift regarding it as a subset of $G$ for every $i \in \NN$. 

Let $\Delta$ be an arbitrary basis for the torsion-free part of the base group $B$. Denote
$$E_i = \mathrm{B}_B(i!, \Delta) $$
for all $i \in \NN$. Therefore    $E_i$ is a centered exhausting Folner sequence in the finitely generated abelian group $B$, see Proposition 	\ref{prop:infty balls are a centered Folner sequence}. The reason for using $i!$ instead of just $i$  will become clear in the proof of Lemma \ref{lem:F_i are transversals} below, c.f Corollary \ref{cor:universal sequence in fg abelian}.

Recall that the subset  $E^Z \subset B^X$ was defined in the obvious way   for a pair of subsets $E \subset B$ and $Z \subset X$, see Equation (\ref{eq:C^Y def}) and Remark \ref{remark:direct product vs sum}.

\begin{lemma}
\label{lem:adapted subsets T_i}
The sequence $T_i = E_i ^{Z_i}$ is adapted to the sequence    $I_i$.
\end{lemma} 
 \begin{proof}
 

 Let    $g \in G$ be a fixed element with    $g =  q n$ where $q \in Q$ and $n\in N$. Let $\Phi \subset N$ be any finite subset.   We are required to show that
 \begin{equation}
 \frac{ | \{ r \in I_i \: : \: \left[g,r \right]  + \Phi \subset T_i    \}  | }{ |I_i| } \xrightarrow{i\to\infty}  1.
 \end{equation}

There is a finite subset   $D \subset B$ and an index $j \in \NN$  such that $\{n, -n\} \cup \Phi \subset D^{Z_{j}}$. 
Recall that $Z_i = \coprod_{j=1}^r I_i x_j$. Combined with  the fact that $I_i$ is centered, this gives
 \begin{equation}
 \label{eq:funky}
 \frac{ | \{ r \in I_i \: : \:  rZ_j \subset Z_i    \}  | }{ |I_i| } \xrightarrow{i\to\infty}  1.
 \end{equation}
Moreover $Z_j \subset Z_i$ for all $i \in \NN$ sufficiently large (in fact, this happens for all $ i \ge j$). Since $E_i$ is an exhausting sequence in the group $B$ we have $D+D+D \subset E_i$ for all $i \in \NN$ sufficiently large.

It follows from Equation (\ref{eq:metabelian identities}) that   the commutator $\left[g,r\right]$ is equal to
\begin{equation}
 \label{eq:the comm}
 \left[g,r\right] =  
\left[q n,r\right] = 
  [n,r] = n - n^r
 \end{equation}
for every element $r \in I_i$. Therefore Equation $(\ref{eq:the comm})$ implies 
\begin{equation}
\label{eq:triple D}
 \left[g,r\right] + \Phi \subset D^{Z_j} + D^{rZ_j} + D^{Z_j} \subset (D+D+D)^{Z_j \cup rZ_j}. 
 \end{equation}

Observe that the two conditions $D+D+D \subset E_i$ and $Z_j \cup rZ_j \subset Z_i$ imply that $\left[g,r\right] + \Phi \subset T_i = B_i^{Z_i}$ according to Equation $(\ref{eq:triple D})$. The result follows from this observation and making use of Equation (\ref{eq:funky}).
\end{proof}


%
%
%


%

\subsection*{Controlled approximations}

Let $H \le G$ be a subgroup with Goursat triplet $\left[H\right] = \left[R, N_H, \alpha_H\right]$. From now on we will denote    $M_{i} = B^{Y_i}$.


\begin{lemma}
\label{lem:the subgroups N_i}
There is a sequence of finite index subgroups $N_i \le N$ satisfying
\begin{equation}
\label{eq:properties of N_i}
Q_i   \le \mathrm{N}_G(N_i), \quad N_H \cap T_{i} = N_i \cap T_{ i} \quad \text{and} \quad N_H \cap M_{i} \le N_i \cap M_{i}
\end{equation}
for all $i \in \NN$.
\end{lemma}
\begin{proof}
  Consider the $Q$-sets $\overline{X}_i = V_i \backslash X$ with the   associated    $Q$-factor maps  $f_i :  X \to  {\overline{X}_i}$. These give rise to  surjective $\ZZ\left[Q\right]$-module homomorphisms
	$ \pi_i : B^X \to B^{\overline{X}_i}$ as in  Lemma \ref{lem:if base is abelian there is a quotient of wreath products}.   Consider the auxiliary permutational   wreath product group
	$$\Gamma_i = Q_i \ltimes B^{\overline{X}_i} $$
  and  its  subgroup
$$\overline{H}_i = Q_i \pi_i(H \cap M_i) \le   \Gamma_i.$$
defined for every $i \in \NN$. The subgroups $\Gamma_i$ are finitely generated, as each $Q_i$, being a finite index subgroup of $Q$, acts on $\overline{X}_i$ with finitely many orbits.

Since the subsets $Y_i \subset X$ are $R$-invariant, we have  $R \le \mathrm{N}_G(H \cap M_i )$  for all $i \in \NN$ according to Proposition \ref{prop:BY for Y invariant is normalized}. In addition,  $V_i$ is acting trivially on the $Q$-factor  set $\overline{X}_i$ and   on $B^{\overline{X}_i}$, see Proposition \ref{prop:kernel acts trivially}. As $Q_i = R \oplus V_i$ by Item (\ref{it:direct sum}) of \ref{lem:controlled Folner sets and transversals}, this shows that the subsets   $\overline{H}_i$ are indeed subgroups.


Take $\overline{T}_i = \pi_i(T_i)$. We may apply Proposition \ref{prop:finite index normal subgroups in a metabelian group} with respect to the subgroup $\overline{H}_i$ and   the finite subset $\overline{T}_i$ of the finitely generated metabelian group  $\Gamma_i$.  This results in   a finite index subgroup  $\overline{N}_i$ of $B^{\overline{X}_i}$ normalized  by $Q_i$, containing the subgroup $\pi_i(H \cap M_i)$ and satisfying
$ \overline{H}_i  \cap \overline{T}_{i} =  \overline{N}_i \cap \overline{T}_i $.

We now pull  back the subgroup $\overline{N}_i$   by letting   $N_i = \pi_i ^{-1}(\overline{N}_i)$ so that the subgroup $N_i$ has finite index in  $N$. It follows from Lemma \ref{lem:if base is abelian there is a quotient of wreath products} that   the subgroups $N_i$ satisfy Equation (\ref{eq:properties of N_i}) as required.
\end{proof}

 
 \begin{lemma}
 \label{lem:V_i acts trivially}
Let $g \in G$ be any element. Then $\left[v,g\right] \in N_i$ for all $ i \in \NN$ and for all elements $v \in V_i$.
 \end{lemma}
 \begin{proof}
We may assume that $g = qm$ for some $q \in Q$ and $m\in N$.  Fix some index $i \in \NN$ and let $v \in V_i$  be any   element.  It follows from Equation (\ref{eq:metabelian identities}) that
 $$\left[v,g\right] = \left[v,qm\right] =   \left[v,m\right].$$
Since $\ker \pi_i \le N_i$, the subgroup $V_i$ lies in the kernel of the $Q$-action on the quotient $N/N_i$, see Proposition \ref{prop:kernel acts trivially}. We conclude that   $\left[v,m\right] \in N_i$ as required.
 \end{proof}

\begin{lemma}
\label{lem:controlled approximation for the subgroup H}
The subgroup $H$ admits a controlled approximation $(K_i, M_i, T_i)$ so that  the finite   index   subgroups $K_i \le G$ admit Goursat triplets
  $\left[K_i \right] = \left[Q_i, N_i, \alpha_i\right] $.
 \end{lemma}

 \begin{proof} 
In light of Equation (\ref{eq:properties of N_i}),   the last remaining step in constructing the controlled approximation $(K_i, M_i,T_i)$ for the subgroup $H$ is to determine the  injective homomorphisms $\alpha_i$. We may apply Corollary \ref{cor:on construction of subgroups using Goursat triplets} and obtain a sequence of injective homomorphisms
		$$\alpha'_i : Q_H \to \mathrm{N}_G(N_i)/N_i$$
		consistent with $\alpha_H$ and defined for all $i \in \NN$ sufficiently large. 
		Let $\alpha_i$ be the extension of each homomorphism $\alpha'_i$ to $Q_i = Q_H \oplus V_i$ determined by the condition     $$\alpha_i(q) = qN_i$$ for all $q \in V_i$. 	The fact that $V_i \le Q_i   \le \mathrm{N}_G(N_i)$ for all $i \in \NN$ implies that the image of each map $\alpha_i$ belongs to $\mathrm{N}_G(N_i)$.
		
To see that every $\alpha_i$ is indeed a group homomorphism, we are required to verify that $\alpha_i(r)$ and $\alpha_i(v)$ commute in the group $\mathrm{N}_G(N_i)/N_i$ for every pair of elements $r \in R$ and $v \in V_i$. This is an immediate consequence of Lemma  \ref{lem:V_i acts trivially}.
%
 \end{proof}
 
 From now on we will refer to the sequence $(K_i, M_i, T_i)$   constructed above as the   \emph{standard controlled approximation} to the subgroup $H$.
 
\begin{lemma}
\label{lem:condition for normalizer}
Consider the standard controlled approximation $K_i$. Let $m \in N$ be any element. If $\left[q,m\right] \in N_i $ for all $q \in Q_i$ then   $m \in \mathrm{N}_N(K_i)$.
\end{lemma} 
 \begin{proof}
Fix an index $i \in \NN$ and an element $g \in K_i$. We are required  to show that $g^m \in K_i$ with respect to the element $m \in N$ as in   the statement of lemma.

The element $g $ can be written   as $g = q(a_q+n)$ where the elements $q \in Q_i$, $n \in N_i$ and $a_q \in N$ satisfy   $\alpha_i(q) = q(a_q + N_i)$.  
It follows from Equation (\ref{eq:metabelian identities}) that
$$ 
g^m = g\left[g,m\right] = g\left[q(a_q+n), m\right] = g \left[q,m\right] = q(a_q + n + \left[q,m\right]), $$
cf. Proposition 	\ref{prop:condition on conjugate to be in H}.
If $\left[q,m\right] \in N_i$ then $g^m \in K_i$ as required.
 \end{proof}

 \begin{lemma}
 \label{lem:containment of normalizers - weak}
 The standard controlled approximation $(K_i, M_i, T_i)$ to the subgroup  $H  $ satisfies 
$ \mathrm{N}_N(H) \cap M_i \le \mathrm{N}_N(K_i) $ for all $i \in \NN$.
 \end{lemma}
 \begin{proof}
Fix an $i \in \NN$ and consider any element   $m \in \mathrm{N}_N(H) \cap M_i$. Taking into account  Lemma \ref{lem:condition for normalizer}, it suffices to show that $\left[q,m\right] \in N_i$ for all $q \in Q_i$.

Consider some fixed element $q \in Q_i$. Write $q = r   v$ for some uniquely determined elements $r \in R = Q_H$ and $v \in V_i$ (we are using a multiplicative notation for the group $Q$).

The fact that  $R \le \mathrm{N}_G(M_i)$ combined with   $m \in M_i$ implies  $\left[r,m\right] \in M_i$. In addition, as $R = Q_H$ there exists an element $h \in H$ with $ h = r l$ for some $ l \in N$. As $ m\in \mathrm{N}_N(H)$ we have from Equation (\ref{eq:metabelian identities}) that
 $$\left[r,m\right] = \left[rl,m\right] = \left[h,m\right] \in H \cap N = N_H.$$
 It follows from Equation (\ref{eq:controlled})  that $\left[r,m\right] \in M_i \cap N_H \le N_i$.  
 
 We know that $[v,m]  \in    N_i$ from Lemma  \ref{lem:V_i acts trivially}. 
 Therefore both commutators $[v,m]$ and $[r,m]$ belong to $N_i$. We may apply   Equation (\ref{eq:standard commutator identities}) and obtain
\begin{equation}
\label{eq:combination commutators}
 [q,m ] = [vr,m] = [v,m]^r + [r,m] \in N_i \le K_i \le \mathrm{N}_N(K_i)
\end{equation} 
as required.  We used the fact that $R \le Q_i \le \mathrm{N}_G(N_i)$ to deduce $[v,m]^r \in N_i$ in the above Equation (\ref{eq:combination commutators}).
 \end{proof}
 
%
 
\subsection*{  Folner sequences  of transversals} 
 
Recall that $H$ is a fixed subgroup of $G$ with $R = Q_H$ and $(K_i,M_i,T_i)$ is     the standard controlled approximation for $H$.

 \begin{lemma}
 \label{lem:F_i are transversals}
  If $\left[N:\mathrm{N}_N(H)\right] < \infty$    
 then $F_i = I_i T_i = I_i E_{i }^{Z_i}$    are finite-to-one transversals of $\mathrm{N}_G(K_i)$ in $G$   for all $i \in \NN$ sufficiently large.
 \end{lemma}
 \begin{proof}
Since $X = \bigcup_i Z_i$ the subgroups $ B^{Z_i}$ exhaust the subgroup   $N$. Therefore  
\begin{equation}
\label{eq:indices}
d = \left[N:\mathrm{N}_N(H)\right] = \left[  B^{Z_i}:\mathrm{N}_N(H) \cap B^{Z_i}\right]
\end{equation}  
for all $i \in \NN$ sufficiently large. From this point  and  until the end of the this proof   let $i \in \NN$ be any index satisfying Equation (\ref{eq:indices}) and with $i \ge d$.

Recall that $E_i = \mathrm{B}_B(i!, \Delta) $ is a universal sequence of transversals in the abelian group $B$.  In fact, the proof of    Corollary  \ref{cor:universal sequence in fg abelian} shows that $T_i = E_i^{Z_i}$ is a finite-to-one transversal to any subgroup of $B^{Z_i}$ whose index divides $i!$.
It is therefore a consequence of Equation (\ref{eq:indices}) that the subset $T_i$ is a finite-to-one transversal to $\mathrm{N}_N(H)$ in $N$. 
As $T_i \subset M_i$, an elementary group theoretic  argument shows that 
$T_i$ is at the same time  a finite-to-one transversal to $\mathrm{N}_N(H) \cap M_i$ in $M_i$.

We have previously established that $\mathrm{N}_N(H) \cap M_i \le \mathrm{N}_N(K_i)$, see Lemma  \ref{lem:containment of normalizers - weak}. Therefore $T_i$ is a finite-to-one transversal to $\mathrm{N}_N(K_i) \cap M_i$ in $M_i$. The normalizer $\mathrm{N}_N(K_i)$ satisfies
$$\ker \pi_i \le N_i  \le \mathrm{N}_N(K_i).$$
where $ \pi_i : B^X \to B^{\overline{X}_i}$ is the $\ZZ\left[Q\right]$-module  homomorphism appearing in the proof of Lemma \ref{lem:the subgroups N_i}. Item (4) of  Lemma \ref{lem:if base is abelian there is a quotient of wreath products} shows therefore that $T_i$ is a finite-to-one transversal to the subgroup $\mathrm{N}_N(K_i)$ in $N$.

As $K_i \le \mathrm{N}_G(K_i)$, the subset $I_i$ is a finite-to-one transversal  to the projection $Q_i$ of $\mathrm{N}_G(K_i) $ to the quotient group $Q$. 
We conclude, relying on Proposition 	\ref{prop:product of finite-to-one is finite-to-one transversal},  that   $F_i = I_i T_i = I_i E_{i}^{Z_i}$ is indeed a finite-to-one transversal to $\mathrm{N}_G(K_i)$ in $G$.
 \end{proof}
 
\begin{lemma}
\label{lem:tiny}
 For any element $n \in N$ and any $\varepsilon > 0$  there is an index $j = j(n,\varepsilon) \in \NN$ so that $T_i = E_i^{Z_i}$ is $(n^q,\varepsilon)$-invariant for all indices $ i > j$ and all elements $q \in Q$ satisfying $n^q \in T_i$.
\end{lemma}
\begin{proof}
Fix some $\varepsilon > 0$. Let   $n = \oplus_{x \in X} b_x$ be any element, where $b_x \in B$.   Recall that $X = \bigcup_i Z_i$. Therefore  there is an index $i_0 \in \NN$ such that $b_x = 0$ for all $x \in X \setminus Z_{i_0}$.  Since $E_i$ is a Folner sequence in the group $B$, we may choose $j = j(n,\varepsilon) $   such that $E_i$ is $(b_x,\frac{\varepsilon}{|Z_{i_0}|})$-invariant for all indices $ i > j$ and all elements $b_x \in B$ with $x \in Z_{i_0}$.   

Consider an element $q \in Q$  satisfying  $n^q \in T_i$ for some index $ i > j$ as in the statement of the Lemma. In particular    $q Z_{i_0} \subset Z_i$. We claim that the set $T_i$ is $(n^q,\varepsilon)$-invariant. Indeed, the size of the symmetric difference $n^q T_i \triangle T_i$ satisfies
$$ |n^q T_i \triangle T_i| \le |E_i|^{|Z_i|-1} \sum_{x \in Z_{i_0}}   |b_{q x} E_i \triangle E_i|    \le |Z_{i_0}| \cdot \frac{\varepsilon}{|Z_{i_0}|} \cdot |T_i| = \varepsilon |T_i|$$
as required.
\end{proof}

 \begin{lemma}
\label{lem:F_i is Folner}
$F_i = I_i T_i = I_i E_i^{Z_i}$ is a Folner sequence in the group $G$.
\end{lemma}
 \begin{proof}

	Fix a constant $\varepsilon > 0$ and an element $g \in G$. Write  $g = q n$ for some elements $q \in Q$ and $ n \in N$.  We will show that the subsets $ F_i $  are $(g,\varepsilon)$-invariant for all $i\in\NN$ larger than some $i_0 = i_0(g,\varepsilon) \in \NN$.


	Let $\varepsilon' > 0$ be a sufficiently small constant   so that $(1-2\varepsilon')^2 >  1-\varepsilon$. 
	Since $I_i$ is a  Folner sequence in the group $Q$  there is an index $i_1 \in \NN$ such that the subset $I_i$ is $(q,\varepsilon')$-invariant for all $i > i_1$.
 Since the sequence $I_i$ is centered, we may   argue exactly as in Lemma \ref{lem:adapted subsets T_i} to find an index $i_2 \in \NN$ such that
	$$ |J_i| > (1-\varepsilon')|I_i| \quad \text{where} \quad J_i = \{r \in I_i \: : \: n^r \in T_i \} $$
for all $i > i_2$.	
 Finally, by Lemma \ref{lem:tiny}    there is some index $i_3 \in \NN$   such that $ T_i$ is $(n^r,\varepsilon')$-invariant provided $n^r \in T_i$ and for all $ i > i_3$.
Take $ i_0 =\max\{i_1,i_2, i_3\}$.

	
We claim that   the subset   $F_i= I_i T_i  $ is indeed $(g,\varepsilon)$-invariant for all  indices $i > i_0$. Fix such an index $ i \in \NN$.
There is a subset $I'_i  \subset I_i$ such that 
	$$ I'_i  \subset q I_i \cap I_i \quad \text{and} \quad |I_i'| > (1-\varepsilon')|I_i|. $$
	For every  element $r \in J_i $ there is a subset $T_i'(r) \subset T_i$ satisfying 
	$$ T'_i(r) \subset ( n^{r} +  T_i) \cap T_i \quad \text{and} \quad |T_i'(r)| > (1- \varepsilon')|T_{i}|. $$
The translate $g F_i $ can be expressed as 
	\begin{align}
	\label{eq:gFi}
	gF_i = g I_i T_i = 
	q n \bigcup_{r\in I_i}r T_i  = 
	q \bigcup_{r\in I_i } r (n^{r} + T_i)  =   \bigcup_{r\in I_i}qr (  n^{r} + T_i).
	\end{align} 
It follows from Equation $(\ref{eq:gFi})$ and from the above discussion that 
	\begin{equation}
	 \bigcup_{r \in I'_i \cap J_i  } r T_i'(r)  \subset g F_i  \cap F_i.
	 \label{eq:strange}
	\end{equation}
Using Equations (\ref{eq:gFi}) and (\ref{eq:strange}) we can estimate   the size of $gF_i \cap F_i$ to be  at least
	$$ |gF_i  \cap  F_i| \ge (1- 2\varepsilon')  |I_i| (1- \varepsilon')|T_i| > (1-2\varepsilon')^2|F_i| > (1-\varepsilon)|F_i|$$
	as required.  \end{proof}
 
 \begin{remark}
The   two   Lemmas \ref{lem:F_i are transversals} and \ref{lem:F_i is Folner}
are to be  compared with  \cite{weiss2001monotilable}, see also Corollary   \ref{cor:residually finite amenable has Foler sequence of finite to one transverals} and Proposition \ref{prop:adapted and uniformly Folner gives product Folner} above.  Weiss' results  are of a similar nature but   more general as they apply to   arbitrary group extensions and provide a universal sequence of transversals. 

However, our two Lemmas \ref{lem:F_i are transversals} and \ref{lem:F_i is Folner}
 are    more precise, in the sense that we do not need to pass to any subsequences whatsoever. This issue is crucial for our purposes, since passing to a subsequence would violate Lemma \ref{lem:adapted subsets T_i}.
\end{remark}

\subsection*{Summary}
Let $R$ be a fixed subgroup of $Q$ and $H \le G$ be any subgroup with $R = Q_H$ and with $[N:\mathrm{N}_N(H)] < \infty$. 
We     put together all of the above  propositions, culminating with a proof of Theorem  \ref{thm:every subgroup admits a controlled approximation} stated in the   beginning of \S\ref{sec:approximation in permutational}.

The main part of the argument is contained in Lemmas \ref{lem:the subgroups N_i}
 and  \ref{lem:controlled approximation for the subgroup H}. This is where we   construct the   controlled approximation $(K_i,M_i,T_i)$ for the subgroup $H$. Here the $K_i \le G$ are finite index subgroups with Goursat triplets $[K_i] = [Q_i, N_i, \alpha_i]$, the subgroups $M_i \le N$ are given by $M_i = B^{Y_i}$ and  the subsets $T_i \subset M_i$ are given by $T_i = E_i^{Z_i} \subset M_i$.

The sequence $F_i = I_i T_i $ is Folner in the group $G$ and consists of finite-to-one transversals of the subgroups $K_i$. This is established in   Lemmas  \ref{lem:F_i are transversals} and
\ref{lem:F_i is Folner}. The sequence of subsets $T_i$ is adapted to the sequence $I_i$ by Lemma \ref{lem:adapted subsets T_i}. 

We have verified all of the necessary requirements for $(K_i,M_i,T_i)$ to be a controlled approximation to the subgroup $H$ with an adapted   sequence of finite-to-one transversals $F_i$, see Definitions \ref{def:controlled approximation} and \ref{def:adapted transversals}. In our case, the subsets $T_i$ and $P_i$ that appear in these two definitions happen to coincide.
 This completes the proof of  Theorem  \ref{thm:every subgroup admits a controlled approximation}. \qed

 \section {Invariant random subgroups of permutational wreath products}

We are ready to present a proof of our main result Theorem \ref{thm:main theorem}. It is restated below for the reader's convinience.

\begin{theorem*}
\label{thm:every IRS is co-sofic}
Let $G$ be a permutational wreath product of two finitely generated abelian groups. Then every invariant random subgroup of $G$ is co-sofic.
\end{theorem*}
\begin{proof}
In light of  Corollary  \ref{cor:enough to prove co-sofic for ergodic} it will suffice to prove the theorem for \emph{ergodic} invariant random subgroups.

Let $\mu \in \IRSerg{G}$  be any ergodic invariant random subgroup of $G$. Consider the map $\pi_* : \IRSerg{G} \to \IRSerg{Q}$, see  Proposition \ref{prop:pushing forward invariant random subgroups}. Since $Q$ is abelian $\pi_* \mu$ is atomic.  In other words, there is some fixed subgroup $R \le Q$ such that   $\mu$-almost every subgroup $H\le G$ satisfies $R = Q_H$. Moreover $[N:\mathrm{N}_N(H)] < \infty$ holds for $\mu$-almost every subgroup $H \le G$    according to 
 Proposition \ref{prop:an element of an invariant random subgroup has finite orbit for N}.

The group $G$ admits a   Folner sequence $F_i$  such that $\mu$-almost every subgroup $H \le G$ has a controlled approximation  with  finite index subgroups $K_i$ and with $F_i$ being an adapted sequence of finite-to-transversals of $\mathrm{N}_G(K_i)$. This is precisely the content of Theorem \ref{thm:every subgroup admits a controlled approximation} applied with respect to the subgroup $R \le Q$. 

The existence of these controlled approximations implies that $\mu$-almost every subgroup $H$ of $G$ is Weiss approximable, as was established in Theorem \ref{thm:a group with a controlled approximation is co-sofic}.  

We conclude that the invariant random subgroup $\mu$ is co-sofic relying on the ergodic theorem for amenable groups, see Theorem \ref{thm:cosofic subgroup implies cosofic IRS}.
\end{proof}

The permutation stability of a  finitely generated permutational wreath product  of two finitely generated abelian groups   follows from the fact that all invariant random subgroups in such a  group are co-sofic and relying on    Theorem \ref{thm:BLT}.


%
%
%

\bibliography{cosofic}
\bibliographystyle{alpha}

\end{document}